\documentclass{article}
\usepackage{authblk}
\usepackage{amsfonts}
\usepackage{amsmath}
\usepackage{amsthm}
\usepackage{txfonts}
\numberwithin{equation}{section} 

\newtheorem{thm}{Theorem}[section]
\newtheorem{defn}[thm]{Definition}
\newtheorem{cor}[thm]{Corollary}
\newtheorem{lemma}[thm]{Lemma}

\newtheorem{rmrk}[thm]{Remark}

\newcommand{\e}{\varepsilon}
\newcommand{\R}{\mathbb{R}}
\newcommand{\N}{\mathbb{N}}
\newcommand{\J}{\mathbb{J}}
\newtheorem{prop}[thm]{Proposition}

\newcommand{\abs}[1]{\left\vert{#1}\right\vert}

\newcommand{\C}{\mathbb{C}}
\newcommand{\ba}{\begin{array}}
\newcommand{\ea}{\end{array}}

\newcommand{\bthm}{\begin{thm}}
\newcommand{\ethm}{\end{thm}}
\newcommand{\bstp}{\begin{stp}}
\newcommand{\estp}{\end{stp}}
\newcommand{\blemma}{\begin{lemma}}
\newcommand{\elemma}{\end{lemma}}
\newcommand{\bprop}{\begin{prop}}
\newcommand{\eprop}{\end{prop}}
\newcommand{\bpf}{\begin{pf}}
\newcommand{\epf}{\end{pf}}
\newcommand{\bdefn}{\begin{defn}}

\newcommand{\edefn}{\end{defn}}
\newcommand{\brk}{\begin{rmrk}}
\newcommand{\erk}{\end{rmrk}}
\newcommand{\bcrl}{\begin{crl}}
\newcommand{\ecrl}{\end{crl}}
\newcommand{\norm}[1]{\left\|#1\right\|}

\newcommand{\beqn}{\begin{equation}}
\newcommand{\eeqn}{\end{equation}}

\renewcommand{\leq}{\leqslant}
\renewcommand{\geq}{\geqslant}

\newcommand{\lm}{\lambda^2}
\newcommand{\mS}{\mathcal{S}}
\newcommand{\mG}{\mathcal{G}}
\newcommand{\mH}{\mathcal{H}}
\newcommand{\mC}{\mathcal{C}}

\newcommand{\mP}{\mathcal{P}}

\newcommand{\beq}{\begin{equation}}
\newcommand{\eeq}{\end{equation}}
\newcommand{\bea}{\begin{eqnarray}}
\newcommand{\tg}{\tilde{\gamma}}
\newcommand{\pmi}{{\bf p}_-}
\newcommand{\ppl}{{\bf p}_+}
\newcommand{\pz}{{\bf p}_0}

\newcommand{\ppm}{{\mathbf{p}_\pm}}
\newcommand{\gb}{\gamma_{\beta}}
\newcommand{\eea}{\end{eqnarray}}

\newcommand{\ob}{\omega_{\beta}}
\newcommand{\YY}{H^1(\mathbb{R}; \mathbb{R}^2)}

\title{A Degenerate Isoperimetric Problem and Traveling Waves to a Bi-stable Hamiltonian System}

\begin{document}
\renewcommand\Authfont{\small}
\renewcommand\Affilfont{\itshape\footnotesize}

\author[1]{Stan Alama\footnote{alama@mcmaster.ca}}
\author[2]{Lia Bronsard\footnote{bronsard@mcmaster.ca}}
\author[3]{Andres Contreras\footnote{acontre@nmsu.edu}}
\author[4]{Jiri Dadok\footnote{dadok@indiana.edu}}
\author[5]{Peter Sternberg\footnote{sternber@indiana.edu}}
\affil[1,2]{Department of Mathematics and Statistics, McMaster University, Hamilton, ON L8S 4K1, Canada}
\affil[3]{Department of Mathematical Sciences, New Mexico State University, Las Cruces, NM 88003}
\affil[4,5]{Department of Mathematics, Indiana University, Bloomington, IN 47405}

\maketitle

\noindent {\bf Abstract:} We analyze a non-standard isoperimetric problem in the plane associated with a metric
having degenerate conformal factor at two points. Under certain assumptions on the conformal factor, we establish
the existence of curves of least length under a constraint associated with enclosed Euclidean area. As a motivation for
and application of this isoperimetric problem, we identify these isoperimetric curves, appropriately parametrized, as traveling
wave solutions to a bi-stable Hamiltonian system of PDE's. We also determine the existence of a maximal propagation speed
 for these traveling waves through an explicit upper bound depending on the conformal factor.

\noindent


\section{Introduction}

We examine here a non-standard type of isoperimetric problem. Given a continuous function $F:\R^2\to [0,\infty)$ vanishing at two points, $\ppl$ and $\pmi$ in the plane,
we seek a curve $\gamma:[0,1]\to\R^2$ that minimizes the distance between these two points in the metric having $F$ as its conformal factor, subject to a constraint associated
with Euclidean area. That is, we seek a solution to the variational problem
\beq
\inf_\gamma E(\gamma)\quad\mbox{with}\quad E(\gamma):=\int_0^1 F(\gamma)\abs{\gamma'}\,dt,\label{isop}
\eeq
where competitors $\gamma:[0,1]\to\R^2$ must satisfy $\gamma(0)=\pmi,\,\gamma(1)=\ppl$ as well as the constraint
\beq
\int_{\gamma} \omega_0=const.\quad\mbox{with}\; \omega_0\;\mbox{given by the $1$-form}\; \omega_0=-p_2dp_1.\label{cons}
\eeq
Since $d\omega_0$ is just the standard Euclidean area form $dp_1\,dp_2$, the isoperimetric nature of the minimizing curve becomes evident--that is, fixing any reference curve,
say $\gamma_0$, joining $\pmi$ to $\ppl$, an application of Stokes' theorem to the closed curve $\gamma-\gamma_0$ reveals that altering the value
of the constant in the constraint $\int_{\gamma} \omega_0=const$ simply amounts to varying the Euclidean area trapped between a competing curve $\gamma$ and the fixed reference curve $\gamma_0.$

What makes this particular isoperimetric
problem non-standard is both the degeneracy of the conformal factor at the two ``wells" $\pmi$ and $\ppl$ and the fact that length is measured with respect to a metric given by $F$
while area is measured with respect to the Euclidean metric. There is a vast literature on isoperimetric problems with assorted assumptions on the conformal factor or ``density,"
though to our knowledge none address this combination of degeneracy and mixture of metrics. We mention, for example, \cite{CJQW,CMV,DDNT,H,RCBM} but of course there are many, many others.

One motivation for our investigation is the connection between such isoperimetric curves--should they exist--and the existence of traveling wave solutions to a
Hamiltonian system associated with the energy functional
\[
H(u):=\int \frac{1}{2}\abs{\nabla u}^2+W(u)\quad\mbox{where}\;W(u)=F^2(u).
\]
The theory of heteroclinic connections to bi-stable gradient-type reaction-diffusion systems in the form of standing or traveling waves
is by now very well-developed in both the scalar and vector-valued settings, including for example, \cite{ABC,AF,AK,AW,FM,M,R,S2} to name but a few studies.
Here, however, rather than seeking traveling wave solutions to a gradient flow $u_t=-\delta H(u)$, we pursue traveling wave solutions
to a Hamiltonian flow associated with $H$, namely,
\beq
\J u_t=\Delta u -\nabla_u W(u)\quad\mbox{where}\quad \J=\left(\begin{matrix} 0 & 1\\ -1& 0\end{matrix}\right)\quad
\mbox{and}\;u=\left(u^{(1)}(x,t),u^{(2)}(x,t)\right),\label{hammy}
\eeq
which conserves the value of $H$ over time. We generally find it more convenient in this approach to view $u$ as $\R^2$-valued but
for $u$ viewed as $\C$-valued, we note that the system would take the form
\[
 iu_t=\Delta u -\nabla_u W(u).
\]
Such a traveling wave solution would follow the ansatz $u=u(x,t)=U(x_1-\nu t)$ for some wave speed $\nu$ with
$U:\R\to\R^2$ then required to solve
\beq
-\nu \,\J U'=U''-\nabla_uW(U)\quad\mbox{on}\;(-\infty,\infty)\quad\mbox{with}\; U(\pm\infty)={\bf p}_{\pm}.\label{hammyode}
\eeq
This problem is itself variational in nature. At least formally, solutions are critical points of
the constrained minimization problem
\begin{eqnarray*}
&&\inf_u H(u)\quad\mbox{among competitors}\;u:\R\to\R^2\;\mbox{satisfying:}\\&&\;U(\pm\infty)={\bf p}_{\pm}\;\mbox{and the constraint}\;
-\int_{\R} u^{(2)} (u^{(1)})'=const.,
\end{eqnarray*}
with the wave speed arising as a Lagrange multiplier associated with the constraint.
It turns out that minimizers of this problem can be found if one identifies minimizers of the isoperimetric problem \eqref{isop},
very much in the spirit of \cite{S1,S2}.

Our paper then is primarily devoted to the study of \eqref{isop}.
Our main result here, stated in Theorem \ref{mainiso}, is the existence of a minimizing curve under certain assumptions on the behavior of $W$ near its minima $\ppl$ and $\pmi.$
We deal with the degeneracy of the conformal factor by first focusing on the problem of finding a constrained minimizer of $E$ joining a non-degenerate point in the plane to either
of the wells of $W$. This one-well isoperimetric problem is solved uniquely for $W$ taking the form of a non-negative quadratic vanishing at the well in Section 2.1,
cf. Theorem \ref{owarea}.  Somewhat surprisingly, the optimal curves are spirals in many cases.

To illustrate how the spiral shape arises, we give a simple derivation of their form for the one-well problem with a radial quadratic potential, $W(z)=|z|^2$, in the unit disk. Suppose we make the assumption that the optimal curve can be parametrized by polar radius $r\in (0,1)$ in the form $z(r)=re^{i\theta(r)}$, and we normalize $\theta(1)=0$. Then the problem reduces to minimizing the length functional
$$   E(z)=\int_0^1 r\sqrt{1 + r^2[\theta'(r)]^2}\, dr \qquad \text{subject to the constraint} \quad
              \int_0^1 \frac{r^2}{2}\theta'(r)\, dr = A,
$$
for given $A$.  By defining $w=z^2(r)=r^2e^{i2\theta(r)}=\rho e^{i\phi(\rho)}$, $\rho\in (0,1)$, the problem is further reduced to minimizing
\begin{gather*}
    L(w):=\int_0^1 |w'(\rho)|\, d\rho=\int_0^1\sqrt{1+\rho^2[\phi'(\rho)]^2}\, d\rho, \quad\text{with} \quad
    \frac14 \int_0^1 \rho\phi'(\rho)\, d\rho = A.
 \end{gather*}
By Jensen's inequality (applied with the convex function $h(x)=\sqrt{1+x^2}$), we obtain an isoperimetric inequality of the form,
$$  L(w)\ge \left(1 + \left[\int_0^1 \rho\,\phi'(\rho)\, d\rho\right]^2\right)^{\frac12} = (1 + 16A^2)^{\frac12}, $$
for any such $w(\rho)=z^2(r)$, with equality holding exactly when $\rho\phi'(\rho)=4A$.  Integrating and returning to the original parametrization, the optimal curve is the spiral
$$  z(r) = r e^{i4A\ln r}, \quad r\in (0,1).  $$
The rigorous derivation of the minimizing curve without the assumption that $\theta=\theta(r)$ and for more general quadratics, done in Section 2.1, follows substantially different methods, using calibrations.

In Section 2.3
we broaden the one-well existence result to cover certain $W$ that are analytic near the well and whose Taylor development begins with the kind of quadratic considered in Section 2.1.
See Theorem \ref{eikonal} for a precise description of the assumptions on $W$. The proofs of these results both come through a type of calibration argument.

In Section 2.2 we present an argument to show that non-degeneracy of the
Hessian of $W$ at its minima is a necessary condition for existence. Put another way, if $W$ is too flat at its minima, then the cost of accumulating length is too cheap in this metric for a given amount of Euclidean area and no minimizer will exist.

We are optimistic that, aside from non-degeneracy of the Hessian of $W$ at $\ppl$ and $\pmi$, these assumptions on the conformal factor can be relaxed and in work presently underway we hope to establish a more general existence result. However, the results obtained in the present article already reveal several important and in some ways unexpected features of the problem. Most crucial to the employment of these curves in the traveling wave context is that their Euclidean arclength is bounded, cf. Propositions \ref{euclid} and \ref{euclid2}, a property not at all obvious to us at the outset, given the degeneracy of the conformal factor $F$. More unexpected is the spiraling nature of the isoperimetric curves when, for example, the conformal factor is purely radial. (See the discussion leading up to Proposition \ref{euclid}.) Such spiraling leads to the existence of spiraling traveling wave solutions to \eqref{hammy}.

In the concluding Section 4 we make explicit the connection between the constrained isoperimetric problem \eqref{isop} and the existence of traveling wave solutions to \eqref{hammy}
satisfying \eqref{hammyode}. Here the main result is Theorem \ref{equivprobs}. Finally, we establish an upper bound for the wave speed associated to these traveling waves. Regarding the traveling wave problem, many crucial questions remain open. For example, can one
make precise the connection between the area constraint value and the wave speed, thereby establishing perhaps the exact interval of attainable $\nu$-values in \eqref{hammyode}?
Also, given the variational nature of the method of solution, do these traveling waves perhaps enjoy some stability property with respect to the Hamiltonian flow \eqref{hammy}?  Are there other non-minimizing traveling waves? These and other
questions remain topics we intend to investigate in the future.

\section{One-well isoperimetric curves via calibrations}

In this section we establish the existence of isoperimetric curves for the case where the degenerate
metric vanishes at one point. The approach involves a type of calibration argument and we first present it for the
case where the conformal factor is given by the square root of a positive definite
quadratic. Then we generalize the argument to the case of certain analytic potentials having
the positive definite quadratic as the leading terms in a Taylor development about the well. We will assume the well is located at the origin, which for one point of degeneracy represents no loss of
generality.
\subsection{The case of a quadratic potential}
To begin, we suppose that $F(p)=\sqrt{W(p)}$ where $W:\R^2\to\R$ is given by the quadratic of the form
\beq
W(p)=p^T\,H_W p.\label{MSA}
\eeq
Here $H_W$ is a constant, real, symmetric, positive definite $2\times 2$ matrix.
We denote by $\lm_1$ and $\lm_2$ the two positive eigenvalues of the matrix $H_W$ and express all points in $\R^2$ and all curves in the plane
with respect to the orthonormal basis $\{v_1,v_2\}$ of eigenvectors of $H_W$. In particular, then $F$ takes the form
\[
F(p)=\sqrt{\lm_1 p_1^2+\lm_2 p_2^2}
\]
and we let
\beq
E(\gamma):=\int_0^1F(\gamma)\abs{\gamma'}\,dt\label{Edefn}
\eeq
for any locally Lipschitz curve $\gamma:[0,1]\to\R^2.$

The argument for existence of a curve minimizing $E$ subject to an area constraint relies on a kind of calibration argument. Before presenting it,
we illustrate the technique on the simpler problem of finding minimizing geodesics connecting the origin to
an arbitrary point $\pz\in\R^2$ in the metric
associated with $E$; that is, we consider the problem without any area constraint.
Introducing the function
\beq\tilde{r}(p):=\frac{1}{2}(\lambda_1\,p_1^2+\lambda_2\,p_2^2)\label{rtilde}
\eeq
and the vector field
\[R(p):=\frac{\nabla\tilde{r}(p)}{F(p)^2}
\]
we now establish:
\bthm\label{genonewell} The unique solution to
\[
\inf\left\{E(\gamma):\;\gamma:[0,1]\to\R^2:\,\gamma\;\mbox{locally Lipschitz},\;\gamma(0)=\pz,\;\gamma(1)={\bf 0}\right\},
\]
is given by the integral curve associated with the vector field $R$.
In other words, the unique minimizer is the curve $z:[0,L_z]\to\R^2$ solving $z'(\ell)=-R(z)$,
$z(0)=\pz$.
\ethm
\begin{rmrk}
The scaling above corresponds to a parametrization by `degenerate arclength' in that $F(z(\ell))\abs{z'(\ell)}=1$ and $L_z=E(z).$
Alternatively, one could follow the integral curve of the linear vector field $F(p)^2R(p)=\left(\lambda_1 p_1,\lambda_2 p_2\right),$
so that writing $\pz=(p_0^{(1)},p_0^{(2)})$ one has $z(t)=\left(p_0^{(1)}e^{-\lambda_1\, t},p_0^{(2)}e^{-\lambda_2\, t}\right)$ with $0\leq t<\infty.$ It is clear from the exponential
decay that the Euclidean arclength of such a geodesic is bounded in terms of $\lambda_1$ and $\lambda_2.$
\end{rmrk}
\begin{rmrk} Note that in the case where $\lambda_1=\lambda_2$, that is in the case of a radial potential, not surprisingly the vector field $R$ is also
radial and the geodesics are just (straight) rays heading into the origin.
\end{rmrk}
\proof
The argument relies on using
$d\tilde{r}$ in the role of a calibration.

Let the curve $z$ be defined as in the statement of the theorem, so that $z'(\ell)=-R(z(\ell))$, with $z(0)=\pz$. We note that all integral curves of $-R$
approach the origin so one has $z(L_z)=(0,0)$. Then we begin by computing
\beq
-\int_{z}d\tilde{r}=\,-\int_0^{L_z} \lambda_1 z_1 z_1'+\lambda_2 z_2 z_2'\,d\ell=\int_0^{L_z}1\,d\ell=L_z=E(z).\label{lz}
\eeq
On the other hand, if $\gamma:[0,L_{\gamma}]\to \R^2$ is any other locally Lipschitz curve starting at $\pz$, ending at the origin and parametrized
by $\ell$, then we have
\begin{eqnarray}
&&-\int_{\gamma}d\tilde{r}=-\int_0^{L_{\gamma}} \lambda_1 \gamma_1 \gamma_1'+\lambda_2 \gamma_2 \gamma_2'\,d\ell\nonumber\\
&&=\,-\int_0^{L_{\gamma}}\left(\lambda_1 \gamma_1,\lambda_2 \gamma_2\right)\cdot\left(\gamma_1',\gamma_2'\right)\,d\ell=
\,-\int_0^{L_{\gamma}}F(\gamma)^2\,R(\gamma)\cdot\gamma'\,d\ell\nonumber\\
&&=\,-\int_0^{L_{\gamma}}F(\gamma)\,R(\gamma)\cdot\frac{\gamma'}{\abs{\gamma'}}\,d\ell\leq  \int_0^{L_{\gamma}}F(\gamma)
\abs{\gamma'}\,d\ell= \int_0^{L_{\gamma}}1\,d\ell=L_{\gamma}=E(\gamma),\label{lg}
\end{eqnarray}
with a strict inequality unless $\gamma'=-R(\gamma)$, i.e. unless $\gamma=z.$ Finally, we note that
\[
\int_{z}d\tilde{r}=\int_{\gamma}d\tilde{r}\quad\mbox{since both curves have the same endpoints}. \]
The minimality of $z$ follows.
\qed
\begin{rmrk}
One upshot of the argument above is that for any $C>0$, all points $p\in\R^2$
lying on the ellipse $\tilde{r}(p)=C$ are equidistant from the origin in the degenerate metric associated with $E$ and are (degenerate) distance $C$
away from the origin, cf. \eqref{lz}.
\end{rmrk}

We are now ready to move on to the one-well isoperimetric problem. To this end, we fix
any real number $A$ and any point $\pz\not=(0,0)$, and introduce the admissible set of curves
\begin{equation}
\mS_{A,\pz}:=\left\{\gamma:[0,1]\to\R^2:\,\gamma\;\mbox{locally Lipschitz},\;\gamma(0)=\pz,\,\gamma(1)=(0,0),\, \mP(\gamma)=A\right\},\label{Sap}\\
\end{equation}
where as before
\[\mP(\gamma):=-\int_0^1\gamma^{(2)}(\gamma^{(1)})'\,dt,\quad
\mbox{and where we write}\quad \gamma=(\gamma^{(1)},\gamma^{(2)}).
\]
We then cast the one-well isoperimetric problem as
\beq
m_0:=\inf_{\mS_{A,\pz}}E(\gamma).\label{purple}
\eeq

Next, introducing two $1$-forms
\beq
\omega_0:=-p_2dp_1\quad\mbox{and}\quad\omega_1:=\frac{1}{\lambda_1+\lambda_2}\left(-\lambda_2 p_2dp_1+\lambda_1 p_1dp_2\right),\label{omone}
\eeq
we observe that
\beq
\mP(\gamma)=\int_{\gamma}\omega_0\quad\mbox{and that}\quad d\omega_0=d\omega_1=dp_1dp_2.\label{sme}
\eeq
We let $\gamma_0$ denote the
 minimizing geodesic starting at $(0,0)$ and ending at $\pz$, as guaranteed by Theorem \ref{genonewell}. Then for any curve $\gamma\in \mS_{A,\pz}$
we conclude from \eqref{sme} that
\[
\int_{\gamma\cup \gamma_0}(\omega_0-\omega_1)=0
\]
and consequently,
\beq
\mP(\gamma)=A\quad\mbox{if and only if}\quad \int_{\gamma}\omega_1=A-C(\gamma_0,\pz)=:\tilde{A}\label{same}
\eeq
where $C(\gamma_0,\pz):=\int_{\gamma_0}(\omega_0-\omega_1).$
We note that the geodesic $\gamma_0$ is simply a convenient choice of reference curve; any curve joining the origin to $\pz$ would do here.

To state the main result of this section we also need to introduce the vector field
\begin{equation}
\Theta(p):=\frac{\left(-\lambda_2 p_2,\lambda_1 p_1\right)}{F(p)^2}
\label{RandT}
\end{equation}
which in the case of equal eigenvalues reduces to simply the $\theta$-direction in standard polar coordinates.

We will prove:

\bthm\label{owarea}
The unique solution to the minimization \eqref{purple} within the class $\mS_{A,\pz}$
is the curve $\gb$ defined as the integral curve of the vector field
\beq
V_\beta(p):=(\cos\beta)\,\Theta(p)-(\sin\beta)\,R(p)\label{Vbeta}
\eeq
that joins $\pz$ to the origin. Here $\beta$ is selected so that
\beq
\frac{\tilde{r}(\pz)}{\lambda_1+\lambda_2}\cot\beta=\tilde{A}.\label{cot}
\eeq
\ethm

We now present the proof of Theorem \ref{owarea}.

\proof
First we note that if we multiply the vector field $V_{\beta}$ by $F^2$ then the corresponding flow, i.e. $\gamma'=F^2(\gamma)V_{\beta}(\gamma)$, is in fact
linear and constant coefficient and takes the form
\begin{equation}
\left(\begin{matrix} \gamma^{(1)}\\ \gamma^{(2)}\end{matrix}\right)'=\Lambda_\beta\left(\gamma^{(1)},\gamma^{(2)}\right):=\left(\begin{matrix}-\lambda_1\sin\beta&-\lambda_2\cos\beta\\
\lambda_1\cos\beta&-\lambda_2\sin\beta\end{matrix}\right) \left(\begin{matrix} \gamma^{(1)}\\ \gamma^{(2)}\end{matrix}\right).\label{linearode}
\end{equation}
One readily checks that the $2\times 2$ matrix on the right has positive determinant and negative trace so either both eigenvalues are real and negative
or they are complex conjugates with negative real part. Either way, all flow lines head into the origin. Of course division by the non-negative factor of $F^2$ simply affects
the parametrization so the same can be said of the integral curves of $V_{\beta}$.

Next we observe that $V_{\beta}$ satisfies the conditions
\beq
V_{\beta}(p)\cdot\Theta(p)=\cos\beta\,\abs{\Theta(p)}^2=\frac{\cos\beta}{F(p)^2}\quad
\mbox{and}\quad V_{\beta}(p)\cdot R(p)=-\sin\beta\,\abs{R(p)}^2=-\frac{\sin\beta}{F(p)^2}
\label{VT}
\eeq
and so since $R\cdot\Theta=0$ one has
\[
\abs{V_{\beta}(p)}=\frac{1}{F(p)}.\]
This implies that the integral curve $\gb$ is parametrized by the degenerate arclength $\ell$, i.e.
\beq
F(\gb(\ell))\abs{\gb'(\ell)}=1.
\eeq
We will denote the degenerate length of $\gb$ by $L_\beta$, so we have $\gb(L_{\beta})=(0,0)$ and $L_{\beta}=E(\gb).$

Now we check that for $\beta$ chosen as in \eqref{cot}, $\gb$ satisfies the area constraint. To this end, we compute that
\begin{eqnarray*}
\int_{\gb}\omega_1&&=\frac{1}{\lambda_1+\lambda_2}\,\int_0^{L_{\beta}}\left(-\lambda_2\gamma_\beta^{(2)},\lambda_1\gamma_\beta^{(1)}\right)\cdot\gamma_\beta'\,d\ell\\
&&=\frac{1}{\lambda_1+\lambda_2}\,\int_0^{L_{\beta}}\left(-\lambda_2\gamma_\beta^{(2)},\lambda_1\gamma_\beta^{(1)}\right)\cdot V_\beta(\gamma_\beta)\,d\ell\\
&&=\frac{1}{\lambda_1+\lambda_2}\cos\beta\,\int_0^{L_{\beta}}\left(-\lambda_2\gamma_\beta^{(2)},\lambda_1\gamma_\beta^{(1)}\right)\cdot \Theta(\gamma_\beta)\,d\ell
=\frac{1}{\lambda_1+\lambda_2}L_\beta\cos\beta.
\end{eqnarray*}
But since
\[
\tilde{r}(\pz)=-\int_{\gb}d\tilde{r}=-\int_0^{L_\beta}\left(\lambda_1\gb^{(1)},\lambda_2\gb^{(2)}\right)\cdot\gb'\,d\ell=
-\int_0^{L_\beta}F^2(\gb)R(\gb)\cdot V_{\beta}(\gb)=\sin\beta\,L_\beta,
\]
we see that indeed $\mP(\gamma)=A$ in light of \eqref{same} and \eqref{cot}. Thus, $\gb\in \mS_{A,\pz}.$

Now let $\ob$ be the $1$-form defined by
\beq
\ob=\cos\beta\,\left(\lambda_1+\lambda_2\right)\omega_1-\sin\beta\,d\tilde{r},\label{ob}
\eeq
cf. \eqref{omone}. The $1$-form $\ob$ will play the role of a calibration in the same manner
 that $d\tilde{r}$ did for the unconstrained problem in the proof of Theorem \ref{genonewell}. Let $\gamma\in\mS_{A,\pz}$ be any competitor in \eqref{purple}. We take $\gamma:[0,L_{\gamma}]\to\R^2$
to be parametrized by degenerate length, i.e. $F(\gamma(\ell))\abs{\gamma'(\ell)}=1$ and denote the angle made between
the vector field $\Theta(\gamma)$ and $\gamma'$ by $\beta_{\gamma}(\ell).$

We observe that
\begin{eqnarray}
&&\int_{\gb} \ob-\int_{\gamma}\ob=\cos\beta\left(\lambda_1+\lambda_2\right)\left(\int_{\gb}\omega_1-\int_{\gamma}\omega_1\right)
-\sin\beta\left(\int_{\gb}d\tilde{r}-\int_{\gamma}d\tilde{r}\right)\nonumber\\
&&=\cos\beta\left(\lambda_1+\lambda_2\right)\,\left(\tilde{A}-\tilde{A}\right)-\sin\beta\,\left((\tilde{r}(0)-\tilde{r}(\pz))-(\tilde{r}(0)-\tilde{r}(\pz))\right)=0.
\label{omeq}
\end{eqnarray}

Then computing the integral of $\ob$ over $\gb$ we find using \eqref{VT} that
\begin{eqnarray}
&&\int_{\gb}\ob=\cos\beta\,\int_{\gb}F^2(\gb)\Theta(\gb)\cdot V_{\beta}(\gb)-\sin\beta\,\int_{\gb}F^2(\gb)R(\gb)\cdot V_{\beta}(\gb)\nonumber\\
&&=\int_0^{L_\beta}\left(\cos^2\beta+\sin^2\beta\right)\,d\ell=L_\beta.\label{gbint}
\end{eqnarray}
Next using that $\abs{\Theta(p)}=\abs{R(p)}=\frac{1}{F(p)}$ we calculate
\begin{eqnarray}
&&
\int_{\gamma}\ob=\cos\beta\,\int_0^{L_{\gamma}}\left(-\lambda_2\gamma^{(2)},\lambda_1\gamma^{(1)}\right)\cdot\gamma'\,d\ell
-\sin\beta\,\int_0^{L_{\gamma}}\left(\lambda_1\gamma^{(1)},\lambda_2\gamma^{(2)}\right)\cdot\gamma'\,d\ell\nonumber\\
&&=\cos\beta\,\int_0^{L_{\gamma}}F^2(\gamma)\Theta(\gamma)\cdot\gamma'\,d\ell
-\sin\beta\,\int_0^{L_{\gamma}}F^2(\gamma)R(\gamma)\cdot\gamma'\,d\ell
\nonumber\\
&&=\cos\beta\,\int_0^{L_{\gamma}}F^2(\gamma)\abs{\Theta(\gamma)}\abs{\gamma'}\cos\beta(\ell)\,d\ell
-\sin\beta\,\int_0^{L_{\gamma}}F^2(\gamma)\abs{R(\gamma)}\abs{\gamma'}\cos{(\beta(\ell)+\pi/2)}\,d\ell\nonumber\\
&&=\int_0^{L_{\gamma}}F(\gamma)\abs{\gamma'}\left(\cos\beta\,\cos\beta(\ell)+\sin\beta\,\sin\beta(\ell)\right)\,d\ell=
\int_0^{L_{\gamma}}\cos\left(\beta-\beta(\ell)\right)\,d\ell\leq L_{\gamma}\label{gmint}
\end{eqnarray}
with equality holding if and only if $\beta(\ell)=\beta$ for all $\ell$, that is, if and only if $\gamma=\gb.$ Combining
\eqref{omeq}, \eqref{gbint} and \eqref{gmint}, we have established that $\gb$ is the unique minimizer.
\qed
\begin{rmrk}If there is no constraint then we are back in the previously solved case of a geodesic which corresponds
here to $\beta=\pi/2.$ In the other extreme, as we choose a larger and larger constraint value $\tilde{A}$ (or equivalently,
as we let $A\to\infty$), then we have $\beta$ approaching $0.$
\end{rmrk}

Given the explicit nature of the system \eqref{linearode}, one can explore the question of how the isoperimetric curve approaches
the origin. We have already noted that the eigenvalues of the constant matrix in that system are either real and negative or complex
with negative real part. Examining this point more closely, one calculates that the eigenvalues, say $\mu_{\pm}$, are given by
\beq
\mu_{\pm}=\frac{-\left(\lambda_1+\lambda_2\right)\sin\beta\pm\sqrt{\left(\lambda_1+\lambda_2\right)^2\sin^2\beta-4\lambda_1\lambda_2}}{2}.
\label{mupl}
\eeq
From this we see that the eigenvalues will be complex if and only if
\beq
\abs{\sin\beta}<\frac{2(\lambda_1\lambda_2)^{1/2}}{\lambda_1+\lambda_2}.\label{spiraling}
\eeq
Referring to \eqref{cot}, a small value of $\sin\beta$ means a large value of $\cot\beta$ and therefore a large
value of the area constraint value $\tilde{A}$, hence of $A$ itself. We conclude that in general, a spiraling isoperimetric
curve solving \eqref{purple} will occur only when the area constraint is sufficiently large. An important exception to this
claim, however, corresponds to the case of a purely radial conformal factor $F$ in which $\lambda_1=\lambda_2.$
In this case the condition \eqref{spiraling} {\it always} holds since the right-hand side equals one. In a similar vein,
if $\abs{\lambda_1-\lambda_2}$ is small then the area constraint threshold beyond which spiraling occurs is small as well.

Pursuing the radial case a bit further, we observe that the system \eqref{linearode} expressed in polar coordinates takes the simple form
\[
r'=-\lambda_1(\sin\beta)\, r,\quad\theta'=\lambda_1\cos\beta
\]
so that in view of \eqref{rtilde} and \eqref{cot} one finds the isoperimetric curve is given by the spiral
\[
\theta(r)=C-\frac{4\tilde{A}}{\abs{\pz}^2}\ln r,\quad 0\leq r\leq \abs{\pz}
\]
for some constant $C$ determined by the argument of $\pz$.

We conclude with a general bound on Euclidean arclength that will be crucial in what follows.
\bprop\label{euclid} There exists a positive constant $L_0$ depending on $A,\,\tilde r(\pz),\,\lambda_1$ and $\lambda_2$ such that
the Euclidean arclength of the minimizing geodesic $\gb$ for \eqref{purple} guaranteed by Theorem \ref{owarea} is bounded by $L_0$.
\eprop
\proof
This follows immediately from the explicit solution to the linear system \eqref{linearode}. One observes that for finite $A$ (hence
finite $\tilde{A}$) and fixed $\pz$ the value of the angle $\beta$ satisfying \eqref{cot} is nonzero. Hence the real part of the eigenvalues
$\mu_{\pm}$ given above are negative and bounded away from zero by a constant depending on $A,\,\pz,\,\lambda_1$ and $\lambda_2$. This means
$\gb$ approaches the zero of $F$ at an exponential rate and the Euclidean arclength bound follows. Regarding the dependence of $L_0$ on $\pz$, we
note that in view of \eqref{cot}, it comes through $\tilde{r}(\pz)$. Since $\tilde{r}(\pz)$ is bounded
above and below by a multiple of $\abs{\pz}^2$, as follows from \eqref{rtilde}, one concludes that in terms of $\pz$, a bound on $L_0$ comes
from a bound on the Euclidean distance from $\pz$ to the well. \qed
\subsection{Nonexistence for more degenerate potentials}\label{ne}

In the previous section we established existence of an isoperimetric curve solving \eqref{purple} under the assumption that the
square of the conformal factor is a non-degenerate quadratic. In this section we illustrate that an assumption of non-degeneracy
is necessary for existence of such an area-constrained minimizer. This, more importantly, also shows this assumption is essential for the existence of minimizers of our main problem \eqref{isop}. We will demonstrate this through the example of a purely radial
potential but it should be clear that non-radial examples can be similarly constructed.

To this end, we take $W:\R^2\to\R$ to be given by $W(p)=\abs{p}^{q'}$ for any $q'>2$. With the conformal factor $F$ given by $\sqrt{W}$ this leads
us to the following version of \eqref{purple}:
\beq
\inf_{\mS_{A,\pz}}E(\gamma)\quad\mbox{where}\quad E(\gamma)=\int_0^1\abs{\gamma}^q\abs{\gamma'}\,dt\label{badradial}
\eeq
for some $q>1$ where $\mS_{A,\pz}$ is again given by \eqref{Sap} for some $A\in\R$. For the sake of concreteness, let us also fix $\pz=(1,0).$ In this radial
setting it is easy to see that without the constraint $\mP(\gamma)=A$, the geodesic as in Theorem \ref{genonewell} is simply given by the
line segment along the $p_1$-axis joining the point $(1,0)$ to the origin, i.e. $\tilde{\gamma}(r):=1-r$ for $0\leq r \leq 1.$ In this unconstrained
setting we compute that the minimal degenerate length is given by $E(\tilde{\gamma})=\frac{1}{1+q}.$

Now consider the following sequence of competitors in $\mS_{A,\pz}$ for the constrained problem. For each positive integer $j$, let $\gamma_j$ consist of the union of the geodesic
 $\tilde{\gamma}$ along with $j$ circles centered at the origin of radius $r_j$ parametrized by $\sigma_j(\theta)=r_j(\cos\theta,\pm\sin\theta)$ for $0\leq\theta\leq 2\pi j$. Noting that $\mP(\tilde{\gamma})=0$, we choose the number $r_j$ so as to satisfy
 the constraint $\mP(\sigma_j)=A$ with the $\pm$ sign determined by the sign of $A$. Thus we have
\[
\mP(\sigma_j)=\pm j\,\pi\,r_j^2=A.
\]
Consequently, one finds that
\begin{eqnarray*}
E(\gamma_j)&&=E(\tilde{\gamma})+E(\sigma_j)=\frac{1}{1+q}+\int_0^{2\pi j}r_j^q\,r_j\,d\theta\\
&&=\frac{1}{1+q}+ 2\pi j\, r_j^{q+1}=\frac{1}{1+q}+2\abs{A}r_j^{q-1}.
\end{eqnarray*}
Since $q>1$, the last term approaches zero as $j\to\infty$, and we see that the infimum of the area-constrained problem is the same as that of the unconstrained problem. However,
it is clear that the line segment $\tilde{\gamma}$ is the only admissible curve yielding an $E$-value of $\frac{1}{1+q}$.  Indeed it is the unique geodesic
 joining $(1,0)$ to $(0,0)$--and since this curve fails to satisfy the constraint,
we conclude that no solution exists for \eqref{purple} in this setting. The flatness of the conformal factor allows for area to be accumulated ``too cheaply" in terms of the
degenerate length.

\subsection{Analytic potentials}

The assumption that $W$ is quadratic is obviously a very strong restriction. In this section
we illustrate one generalization, still based on a calibration argument, for the case of certain analytic
potentials to be described below. To this end, we now suppose that
$F$ in \eqref{Edefn} is given by $F=\sqrt{W}$ where $W:\R^2\to\R$ is real analytic in $p=(p_1,p_2)$ in some neighborhood of the origin. In light of the
necessity of nondegeneracy revealed in Section \ref{ne}, we must assume that in a coordinate system relative to the
orthonormal basis of eigenvectors of $D^2W(0,0)$ one has the Taylor development
\beq
W(p_1,p_2)=\lm_1 p_1^2+\lm_2 p_2^2+O\left(\abs{p_1}^3+\abs{p_2}^3\right)\quad\mbox{for}\;\lm_1,\,\lm_2>0.\label{Taylor}
\eeq

In what follows we invoke multi-index notation so that
for $\alpha=(\alpha_1,\alpha_2)$ with $\alpha_1$ and $\alpha_2$ non-negative integers we write $p^{\alpha}=p_1^{\alpha_1}p_2^{\alpha_2}$ and $\abs{\alpha}:=\alpha_1+\alpha_2$.

We will prove
\bthm\label{eikonal}  Assume $W:\R^2\to [0,\infty)$ is real analytic in a ball $B$ centered at the origin and for every
$\beta\in (0,\pi)$ suppose $W$ can be expressed as
\beq
W(p)=\abs{\nabla g_\beta(p) +\Lambda_{\beta}(p)}^2\quad\mbox{for all}\;p\in B\quad \label{Wexp}
\eeq
 where $\Lambda_{\beta}$ is given in \eqref{linearode},
and
$g_\beta:B\to\R$  is analytic in $B$ with a Taylor series
of the form
\[
g_\beta(p)=\sum_{\abs{\alpha}\geq 3}a_{\alpha}p^\alpha.
\]
Let $\gamma_0$ be the (unique) minimizing geodesic joining $\pz$ to $0$ and again let $\mS_{A,\pz}$ denote the admissible set
given by \eqref{Sap}.
Then there exists an $R>0$ with $B_R(0)\subset B$ such that for
 every $\pz\in B_R(0)$
there exists a solution to
\beq
m_0:=\inf_{\mS_{A,\pz}}E(\gamma)\label{purple2}
\eeq
for every $A$ in an open interval $I(\pz)$ containing $\mP(\gamma_0)$

\ethm

\begin{rmrk} {\it One class of potentials $W$ where condition \eqref{Wexp} is always solvable is the case where $W$ is analytic
and radial. For example, if $W(p)=\abs{p}^2+f(\abs{p})$ where $f(r)=a_1r^4+a_2r^6+\ldots$ then a straight-forward calculation
goes to show that the choice
\[
g_\beta(r)=\int_0^r\left\{(\sin\beta)\,s-\sqrt{\sin^2\beta\,s^2+f(s)}\right\}\,ds
\]
yields a (local) solution to $\eqref{Wexp}$. We emphasize, however, that in this general radial case, letting $\beta$ range between $0$ to $\pi$ may not always yield a corresponding
interval of allowable constraint values $I(\pz)$ consisting of all of $\R$.}
\end{rmrk}

\begin{rmrk}
Note that in the case where $g_\beta\equiv 0$, since $\abs{\Lambda_\beta(p)}^2=\lm_1 p_1^2+\lm_2 p_2^2$, Theorem \ref{eikonal} reduces to the quadratic case of Theorem \ref{owarea}.
\end{rmrk}
\begin{rmrk}
We observe that if one expands the right-hand side of \eqref{Wexp}, then this condition relating $W$ to $g_\beta$ can
be phrased as an eikonal-type equation for $g_\beta$ of the form
\beq
W(p)-\lm_1 p_1^2-\lm_2 p_2^2=L(\nabla g_\beta)+\abs{\nabla g_\beta}^2\label{linlin}
\eeq
where $L$ is given by the linear operator $L(\nabla g_\beta):=2\langle \Lambda_\beta,\nabla g_\beta\rangle.$ We suspect that  for $W$ analytic, non-negative
and having a non-degenerate minimum of zero at the origin, \eqref{Wexp}
can always be locally solved for $g_\beta$ for every $\beta$ between $0$ and $\pi$ but this is not clear to us. What we can say at this point is that one can generate
a formal power series for $g_\beta$; that is, one can determine the Taylor coefficients $a_{\alpha}$ in terms of the Taylor coefficients
 of $W$, but the convergence of
the series remains undetermined. To describe the algorithm for determining the coefficients, suppose we write
\[
W(p_1,p_2)=\lm_1 p_1^2+\lm_2 p_2^2+\sum_{\abs{\alpha}\geq 3}b_{\alpha}p^\alpha
\]
for given coefficients $b_{\alpha}$ and for each integer $n\geq 3$, let $P_n$ be the homogeneous
polynomial of degree $n$ given by $\sum_{\abs{\alpha}=n}a_\alpha p^{\alpha}$ and let $Q_n$ be the homogeneous
polynomial of degree $n$ given by $\sum_{\abs{\alpha}=n}b_\alpha p^{\alpha}$. Then starting from \eqref{linlin} we find
\beq
L(\nabla P_n)=Q_n-\sum_{j+k=n+2}\langle \nabla P_j,\nabla P_k\rangle.\label{recur}
\eeq

Note that $L$ is invertible for $P_n=$ homogeneous polynomial since the condition $L(\nabla P_n)=0$ says
$P_n$ is constant along the integral curves
of the vector field $\Lambda$ used to solve the quadratic case. Since these curves all emanate from the origin and
$P_n(0,0)=0$ this forces $P_n\equiv 0.$

In this way we can formally recursively generate all of the $a_\alpha's$, in particular showing that (up to an additive constant),
$g_\beta$--if it exists in the analytic setting--is unique.
\end{rmrk}
\begin{rmrk} Even if one could prove that such an analytic function $g_\beta$ satisfying \eqref{linlin} always exists, the assumption of analyticity on $W$
still seems much too strong. For this reason we have not pursued the question of convergence of this formal series solution very strenuously.
In work in progress we are pursuing a completely different approach aimed at asserting an existence result
such as Theorem \ref{eikonal} under much milder assumptions on $W$.
\end{rmrk}
\proof
Fix $\pz$ with norm sufficiently small so that $\pz$ lies inside the ball, say, $\frac{1}{2}B$. Then for any fixed $\beta\in (0,\pi)$
define the vector field
$\tilde{V}_\beta:B\to\R^2$--a generalization of $V_{\beta}$ given in \eqref{Vbeta}-- via
\beq
\tilde{V}_\beta=\frac{1}{W}\left(\nabla g_\beta+\Lambda_\beta\right).\label{Vfield}\eeq
We note that since $\nabla g_\beta$ is of quadratic
order while $\Lambda_\beta$ is linear, it is still the case, as it was for $V_{\beta}$, that in a neighborhood of the origin all integral curves of $\tilde{V}_\beta$ lead into the origin.
Now let $\gamma_\beta:[0,L_{\gamma_\beta}]\to\R^2$ denote the unique solution to
\beq\gamma_\beta'(\ell)=\tilde{V}_\beta(\gamma_\beta(\ell)),\quad \gamma(0)=\pz.\label{vf}
\eeq
In light of \eqref{Wexp} we see that
\[
F(\gamma_\beta(\ell))\abs{\gamma_\beta'(\ell)}\,d\ell=1\quad\mbox{for all}\;\ell\in (0,L_{\gamma_\beta})\quad\mbox{and so}\;L_{\gamma_\beta}=E(\gamma_\beta).\]
We claim that $\gamma_\beta$ solves \eqref{purple2} for the value $A=\mP(\gamma_\beta).$

To this end, define the $1$-form $\tilde{\omega}_\beta$ via the formula
\[
\tilde{\omega}_\beta:=\ob+dg_\beta,\]
with $\ob$ given as before by \eqref{ob}.

Let $\gamma_1:[0,L_1]\to\R^2$ be any other competitor in $\mS_{A,\pz}$, also parametrized by $\ell$, so that $L_1=E(\gamma_1)$. Then an easy calculation gives
\[
\int_{\gamma_\beta}\tilde{\omega}_\beta=\int_0^{L_{\gamma_\beta}}\left(\nabla g_\beta(\gamma_\beta)+\Lambda_\beta(\gamma_\beta)\right)\cdot\gamma_\beta'\,d\ell=
\int_0^{L_{\gamma_\beta}}F(\gamma_\beta)\abs{\gamma_\beta'}\,d\ell=L_{\gamma_\beta},\]
while a similar calculation gives
\[
\int_{\gamma_1}\tilde{\omega}_\beta=\int_0^{L_1}\left(\nabla g_\beta(\gamma_1)+\Lambda_\beta(\gamma_1)\right)\cdot\gamma_1'\,d\ell\leq
\int_0^{L_1}F(\gamma_1)\abs{\gamma_1'}\,d\ell=L_1,\]
with equality holding iff $\gamma_1'=\tilde{V}_\beta(\gamma_1)$, i.e. iff $\gamma_1=\gamma_\beta.$ Furthermore,
since $\gamma_\beta$ and $\gamma_1$ share the same endpoints one has that $\int_{\gamma_\beta}dg_\beta=\int_{\gamma_1}dg_\beta$ and since $\mP(\gamma_\beta)=\mP(\gamma_1)$
we conclude that $\int_{\gamma_\beta} \tilde{\omega}_\beta=\int_{\gamma_1}\tilde{\omega}_\beta$
by the same reasoning as that which led to \eqref{omeq}. It follows that $\gamma_\beta$ uniquely solves \eqref{purple2} among competitors sharing its
constraint value.

Finally, we note that the minimizing (unconstrained) geodesic $\gamma_0$ corresponds to the integral curve of the vector field $\tilde{V}_\beta$ for $\beta=\pi/2$ in the same
manner as was shown in Theorem \ref{genonewell} for the purely quadratic case.
Hence, letting $\beta$ range between $0$ and $\pi$, and noting the continuous dependence of the solution $\gamma_\beta$ to
\eqref{vf} on $\beta$, we generate an interval $I(\pz)$ of constraint values $A$ containing $\mP(\gamma_0)$ for which \eqref{purple2} is solvable.
\qed
\vskip.2in
Since the linear part of the vector field $F^2V_\beta$ is $\Lambda_\beta$, one still has exponential approach of the minimizing isoperimetric curve to the well.
Consequently, one obtains the analog of Proposition \ref{euclid} in this more general setting:
\bprop\label{euclid2} There exists a positive constant $L_0$ depending on $A,\,\pz,\,\lambda_1$ and $\lambda_2$ such that
the Euclidean arclength of the minimizing geodesic for \eqref{purple2} guaranteed by Theorem \ref{eikonal} is bounded by $L_0$.
\eprop

\section{The two-well isoperimetric problem}
\subsection{Existence of a minimizing curve}

 The analysis in the previous section allows us to find an optimal curve for the area constrained problem joining a non-degenerate point to a zero of the conformal factor under
 various circumstances. It remains to use this to prove the existence of a geodesic to our main problem, namely \eqref{isop}.
In this section we show how to establish the existence of a minimizing isoperimetric curve joining two points of degeneracy for the conformal
 factor via a direct method, provided one is able to solve the one-well problem of the previous section for all constraint values.

To this end, we make the following assumptions on $W:=F^2$. We assume $W:\R^2\to [0,\infty)$ is continuous and vanishes at precisely
two distinct points $\ppl$ and $\pmi$. We assume the existence of two disjoint balls $B_+$ and $B_-$ centered at $\ppl$ and $\pmi$ respectively,
such that within each ball $W$ is smooth with a positive definite Hessian at $\ppl$ and $\pmi$. Most crucially, we assume:
\vskip.1in
\noindent{\it {\bf Assumption A.} For every point $\pz$ in either $B_+\setminus \{\ppl\}$ or $B_-\setminus \{\pmi\}$ and for every real number $A$, there exists a minimizer to the one-well problem:
\[
\inf E(\gamma)\quad\mbox{where as before}\; E(\gamma):=\int_0^1 F(\gamma)\abs{\gamma'}\,dt,
\]
and the infimum is taken over all Lipschitz continuous curves $\gamma:[0,1]\to B_+$ (resp. $B_-$) such that $\gamma(0)=\pz,\;\gamma(1)=\ppl$ (resp. $\pmi$) and such that
$\mP(\gamma)=A.$ Furthermore, we assume the Euclidean arclength of this minimizer can be bounded by a constant depending on $W,\,\pz$ and $A$.}
\vskip.1in

We point out that the above assumption has been verified for the case of $W$ given by a non-degenerate quadratic in $B_+$ and $B_-$ through Theorems \ref{owarea} and \ref{euclid},
and it holds for $W$ analytic in the balls that satisfy the condition \eqref{Wexp} through Theorem \ref{eikonal} and \ref{euclid2}, provided the interval $I(\pz)$ arising in Theorem \ref{eikonal} consists of all of $\R$ for all
$\pz\in B_+\setminus \{\ppl\}$ and all $\pz\in B_-\setminus \{\pmi\}$.

Finally, we make the following assumption about the continuous function $F$ outside the two balls: we assume that for some number $m_0>0$ one has the lower bound
\beq F(p)\geq m_0\quad\mbox{for all}\; p\in \R^2\setminus \left(B_+\cup B_-\right).\label{lb}
\eeq
We observe that under the nondegeneracy assumptions made above on $W$ near its minima, it also follows that
\beq
F(p)\geq c_0\abs{p-{\bf p}_{\pm}}\quad\mbox{for all}\;p\in B_{\pm}\label{lbound}
\eeq
for some $c_0>0.$

Now for any real number $A_0$, we let
\beq
\mS_{A_0}:=\big\{\gamma:[0,1]\to\R^2:\,\gamma\;\mbox{locally Lipschitz},\;\gamma(0)={\bf p}_-,\;\gamma(1)={\bf p}_+,\; \mP(\gamma)=A_0\big\},
\label{Sapo}
\eeq
where again
\[\mP(\gamma):=-\int_0^1\gamma^{(2)}(\gamma^{(1)})'\,dt,\quad
\mbox{with}\quad \gamma=(\gamma^{(1)},\gamma^{(2)}).
\]

\bthm\label{mainiso} Under the assumptions above, for any $A_0\in\R$ there exists a minimizer $\gamma_*\in \mS_{A_0}$ to the problem
\beq
m_*:=\inf_{\mS_{A_0}}E(\gamma).\label{purple1}
\eeq
\ethm
\proof Let $\{\gamma_j\}\subset\mS_{A_0}$ denote a minimizing sequence so that $E(\gamma_j)\to m_*$ as $j\to\infty.$
We will argue that we can replace this sequence by a modified sequence that is still a minimizing sequence in $\mS_{A_0}$ but
which has uniformly bounded Euclidean arclength. Once this is achieved it will follow fairly easily that the direct method
yields the existence of a minimizer.
\vskip.2in
\noindent 1. Uniform bound on Euclidean arclength of a minimizing sequence.\\

To this end, we denote by $t_j^{(1)}$ the smallest time such that $\gamma_j(t_j^{(1)})\in\partial B_-$ so that, in particular,
$\gamma_j(t)\in B_-$ for all $t\in [0,t_j^{(1)}).$ Writing ${\bf p}_-=(p_-^{(1)},p_-^{(2)})$ we appeal to \eqref{lbound} to see that
\begin{eqnarray}
\abs{\mP\left(\gamma_j([0,t_j^{(1)}])\right)}&&:=\abs{-\int_0^{t_j^{(1)}} \gamma_j^{(2)}(\gamma_j^{(1)})'\,dt}=
\abs{\int_0^{t_j^{(1)}} \left(\gamma_j^{(2)}-p_-^{(2)}\right)(\gamma_j^{(1)})'\,dt+p_-^{(2)} \int_0^{t_j^{(1)}}(\gamma_j^{(1)})'\,dt}
\nonumber\\
&&\leq \int_0^{t_j^{(1)}} \abs{\gamma_j^{(2)}-p_-^{(2)}}\abs{\gamma_j'}\,dt+\abs{p_-^{(2)}}\abs{\gamma_j^{(1)}(t_j^{(1)})-p_-^{(1)}}\nonumber\\
&&\leq \frac{1}{c_0}E(\gamma_j)+\abs{{\bf p_-}}R_-\leq \frac{ m_*}{c_0}+\abs{{\bf p_-}}R_-+1=:K_-
\label{areaone}
\end{eqnarray}
We then perform our first modification on $\{\gamma_j\}$ by replacing the initial part of the curve $\gamma_j$, namely $\gamma_j([0,t_j^{(1)}])$, by the solution
to the one-well problem \eqref{purple} with the constraint value $A$ given by $\mP\left(\gamma_j([0,t_j^{(1)}])\right)$ and the role of $\pz$ taken up by $\gamma_j(t_j^{(1)}).$
 We then perform a similar modification on the final portion of $\gamma_j$ to enter the ball
$B_+,$ solving the one-well isoperimetric problem in $B_+$ subject to a constraint value bounded by say $K_+$, in analogy with \eqref{areaone}.
By this procedure we create a new curve that still lies in $\mS_{A_0}$  and this modification only reduces the value of $E$ so the modified sequence is still a minimizing sequence.

Now if each element of the minimizing sequence only enters the ball $B_-$ during this first time interval and only enters $B_+$ during one time interval at the end
(i.e. with right endpoint $1$), then we easily obtain a uniform arclength bound on the modified minimizing sequence as follows. For the time interval during which $\gamma_j$ resides inside $B_-$ or inside $B_+$ we are assuming via Assumption A that the one-well solution satisfies a uniform bound on Euclidean arclength--again, an assumption that is verified in
Proposition \ref{euclid} for quadratic potentials or Proposition \ref{euclid2} for appropriate analytic potentials. For the (one) time interval
 during which $\gamma_j$ lies in the complement of $B_-\cup B_+$, we use assumption \eqref{lb} on $F$ to conclude that
\beq
\int_{\{t:\,\gamma_j(t)\not\in\,B_-\cup B_+\}}\abs{\gamma_j'}\,dt\leq \frac{1}{m_0}\int_{\{t:\,\gamma_j(t)\not\in\,B_-\cup B_+\}} F(\gamma_j)\abs{\gamma_j'}\,dt\leq
\frac{1}{m_0}E(\gamma_j)<\frac{1}{m_0}(m_*+1).\label{easybd}\eeq
Combining these uniform bounds in $B_-,\,B_+$ and $\R^2\setminus (B_-\cup B_+)$, we obtain a bound on the total Euclidean arclength of $\gamma_j$ that is independent of $j$.

It certainly could happen, however, that $\gamma_j$ re-enters $B_-$ (resp. $B_+$) other than during the initial (resp. final) time interval.
In this case, further modifications are required. Note that besides the initial arc, any parts of the curve $\gamma_j$ that enter $B_-$ must consist
of a union of curves both starting and ending on $\partial B_-.$ Let $\Gamma_j^{(1)}$ denote the curves in this union that never get closer than Euclidean
distance $1/2$ to ${\bf p}_-$. Then let $\Gamma_j^{(2)}$ denote the rest, namely those curves in $B_-$ with endpoints on $\partial B_-$ that do get closer than distance $R_-/2$.

The curves in $\Gamma_j^{(1)}$ do not need to be modified since a bound on their Euclidean arclength comes in a manner similar to \eqref{easybd}, utilizing \eqref{lbound} to see that
$F(\gamma_j)\geq \frac{c_0\,R_-}{2}$ for these curves:
\beq
\int_{\Gamma_j^{(1)}}\abs{\gamma_j'}\,dt\leq\frac{2}{c_0\,R_-} \int_{\Gamma_j^{(1)}}F(\gamma_j)\abs{\gamma_j'}\,dt\leq \frac{2}{c_0\,R_-}E(\gamma_j)\leq \frac{2}{c_0\,R_-}(m_*+1).\label{easy2}
\eeq

Regarding the curves in $\Gamma_j^{(2)}$, we first note that the number of curves in this union is bounded independent of $j$. This is because for each curve in this union,
say $\tilde{\gamma}:[a,b]\to \overline{B_-}$,  one has a lower bound on its contribution to the total energy, namely
\begin{eqnarray}
\int_a^b F(\tilde{\gamma})\abs{\tilde{\gamma}'}\,dt&&\geq \int_{[a,b]\cap\,\{t:\,R_-/2\leq \abs{\tilde{\gamma}(t)-\pmi}\leq R_-\}}F(\tilde{\gamma})\abs{\tilde{\gamma}'}\,dt\nonumber\\
&&\geq\frac{c_0\,R_-}{2} \int_{[a,b]\cap\,\{t:\,R_-/2\leq \abs{\tilde{\gamma}(t)-\pmi}\leq R_-\}}\abs{\tilde{\gamma}'}\,dt\geq \frac{c_0\,R_-^2}{2}.\label{finite}
\end{eqnarray}
Thus, since say $E(\gamma_j)\leq m_*+1$, there can be at most $\frac{2(m_*+1)}{c_0\,R_-^2}$ curves in $\Gamma_j^{(2)}.$

Now we introduce a number $r_j$ to denote the closest any curve in $\Gamma_j^{(2)}$ gets to ${\bf p}_-$, i.e.
\[
r_j:=\min_{\gamma\in\Gamma_j^{(2)}}\left(\min_t\abs{\gamma(t)-{\bf p}_-}\right).\]
If $\limsup_{j\to\infty} r_j=:\overline{r}>0$ then again no further modifications are needed, since then we can obtain a uniform upper bound on the
total Euclidean arclength of all curves in $\Gamma_j^{(2)}$ as in \eqref{easy2} with the factor of $\frac{2}{R_-}$ replaced by say a factor of $\frac{1}{\overline{r}}.$

If, however, $r_j\to 0$, we will modify the element of $\Gamma_j^{(2)}$ that passes closest to ${\bf p}_-$ as follows: First we view this element, say $\gamma:[a,b]\to \overline{B_-}$
as the union of two curves, say $\gamma_1$ starting on $\partial B_-$ and
ending at a point $r_j$ distance from ${\bf p}_-$ and $\gamma_2$ starting at this closest point and ending (presumably) somewhere else on $\partial B_-$.
If for any $j$, $r_j=0$ then the `offending'
curve $\gamma$ must pass through ${\bf p}_-$ itself and we simply replace $\gamma_1$ by a minimizing isoperimetric curve guaranteed by Assumption A starting at $\gamma(a)\in\partial B_-$, ending at ${\bf p}_-$ with constraint value
$A=\mP(\gamma_1).$ Similarly, we replace $\gamma_2$ by the isoperimetric curve starting at ${\bf p}_-$, ending at $\gamma(b)\in\partial B_-$, subject to area constraint $A=\mP(\gamma_2).$
This replacement can only lessen the value of $E$ while preserving the global condition $\mP=A_0$ so we still have a minimizing sequence lying in $\mS_{A_0}$. Furthermore, the total Euclidean arclength
of the two new curves is bounded by a constant independent of $j$.

The replacement procedure in the case where $r_j\to 0$ but a given $r_j$ is positive is slightly different. In this case, denote by ${\bf q}_j$ a point on $\gamma$ closest to ${\bf p}_-$ (so that
$\abs{{\bf q}_j-{\bf p}_-}=r_j$) and as above
split $\gamma$ into $\gamma_1$ (coming into ${\bf q}_j$ from $\partial B_-$) and $\gamma_2$ (going out of ${\bf q}_j$ to $\partial B_-$).
Then augment $\gamma_1$ by the (directed) line segment joining ${\bf q}_j$ to ${\bf p}_-$ and augment $\gamma_2$ by the directed line segment joining ${\bf p}_-$ to ${\bf q}_j.$
Call these two curves with added line segments $\tilde{\gamma}_1$ and $\tilde{\gamma}_2.$ Note that
\[
\mP(\tilde{\gamma}_1)+\mP(\tilde{\gamma}_2)=\mP(\gamma_1)+\mP(\gamma_2)=\mP(\gamma),
\]
due to the cancellation of the `areas' of the two oppositely directed line segments.
Then, as in the previous case, we replace $\tilde{\gamma}_1$ and $\tilde{\gamma}_2$ by isoperimetric minimizers joining $\gamma(a)$ to ${\bf p}_-$ and ${\bf p}_-$ to $\gamma(b)$ respectively
 and enclosing `areas' $\mP(\tilde{\gamma}_1)$ and $\mP(\tilde{\gamma}_2).$ These two curves must have $E$-values less than or equal to $\tilde{\gamma}_1$ and $\tilde{\gamma}_2$ respectively and as always have uniformly bounded Euclidean arclength,
 but due to the added line segments, their $E$-values may exceed $E(\gamma_1)$ and/or $E(\gamma_2).$ However, since by \eqref{MSA} or \eqref{Wexp},
 $F^2$ is quadratic to leading order we have
 \[
 E(\mbox{line segment})\leq C\int_0^{r_j}r\,dr=O(r_j^2).
 \]
Thus, since $r_j\to 0$, the extra energy of the modified sequence approaches zero. Hence, after this modification we still have a minimizing sequence in $\mS_{A_0}$.

 Having replaced this curve in $\Gamma_j^{(2)}$ that passed closest to ${\bf p}_-$ by a pair of one-well isoperimetric curves through $\pmi$ (for which we have a Euclidean arclength bound), we remove this curve from the collection $\Gamma_j^{(2)}$ and now check again whether we still have $\overline{r}:=\limsup_{j\to\infty}r_j= 0$.
 If now $\overline{r}>0$, we  have our uniform Euclidean bound as in \eqref{easy2} with the factor of $\frac{2}{R_-}$ replaced by say a factor of $\frac{1}{\overline{r}}.$ If $\overline{r}$ is still zero,
 we repeat this procedure, at most $\frac{2(m_*+1)}{c_0\,R_-^2}$ times for every element of $\Gamma_j^{(2)}$ until, if necessary, every offending element of $\Gamma_j^{(2)}$ has been replaced by a pair of one-well minimizers.

We perform the same modification procedure, as necessary, in a neighborhood of $B_+$. In light of \eqref{easybd}, \eqref{easy2} and the above discussion, the result is a minimizing sequence in $\mS_{A_0}$
that satisfies a uniform bound on its Euclidean arclength.
\vskip.2in
\noindent
2. Existence via direct method.
\vskip.1in
Having established the existence of a minimizing sequence (still denoted by $\{\gamma_j\}$) satisfying a uniform bound
\beq
\int_0^1\abs{\gamma_j'(t)}\,dt=:L_j\quad\mbox{where}\quad L_j\leq L_*\label{unifEB}
\eeq
for some constant $L_*$ independent of $j$, we now reparametrize these curves to have constant speed $L_j$.
Thus, denoting the new parameter by $\tau$ (but still using the notation $\gamma_j$) we now have
\beq
\gamma:[0,1]\to\R^2\quad\mbox{satisfying}\quad\abs{\gamma_j'(\tau)}=L_j\;\mbox{for a.e.}\;\tau\in(0,1).\label{unitspeed}
\eeq
Integrating this condition, note that we also have a uniform bound on the modulus of $\gamma_j$ of the form
\[
\abs{\gamma_j(\tau)}\leq \abs{{\bf p_-}}+L_*\quad\mbox{for all}\;\tau\in [0,1].
\]
Putting these two bounds together, we obtain, in particular, the uniform bound
\[
\norm{\gamma_j}_{W^{1,\infty}((0,1);\R^2)}\leq C_0\quad\mbox{for some constant}\;C_0\;\mbox{independent of}\;j.
\]
This provides us with enough compactness to proceed with the direct method. After passing to
subsequences (not denoted here), one has the existence of a Lipschitz map which in particular satisfies $\gamma_*\in H^1((0,1);\R^2)$ such that
\beq
\gamma_j\to \gamma_*\;\mbox{uniformly on}\;[0,1],\quad \gamma_j'\rightharpoonup\gamma_*'\;\mbox{weakly in}\;L^2((0,1);\R^2).
\label{convergence}
\eeq
In light of \eqref{convergence}, it is immediate that $\gamma_*(0)={\bf p}_-$ and $\gamma_*(1)={\bf p}_+$ and the product
of strongly and weakly convergent sequences also readily yields that
\[
\lim_{j\to\infty}\int_0^1\gamma_j^{(2)}(\tau)(\gamma_j^{(1)})'(\tau)\,d\tau=
\int_0^1\gamma_*^{(2)}(\tau)(\gamma_*^{(1)})'(\tau)\,d\tau
\]
so $\mP(\gamma_*)=\lim_{j\to\infty}\mP(\gamma_j)=A_0$.
Hence $\gamma_*\in \mS_{A_0}$.

It remains to check the lower-semi-continuity of $E$ under these convergences. For this let us write
\[
E(\gamma_j)=\int_0^1\left(F(\gamma_j)-F(\gamma_*)\right)\abs{\gamma_j'}\,d\tau+\int_0^1 F(\gamma_*)\abs{\gamma_j'}\,d\tau.
\]
Clearly, \eqref{unifEB}, \eqref{unitspeed} and \eqref{convergence} imply that
\[
\lim_{j\to\infty}\abs{\int_0^1\left(F(\gamma_j)-F(\gamma_*)\right)\abs{\gamma_j'}\,d\tau}\leq
L_*\lim_{j\to\infty}\int_0^1\abs{F(\gamma_j)-F(\gamma_*)}\,d\tau=0
\]
since $F$ is continuous,
so we have \[
\liminf_{j\to\infty}E(\gamma_j)=\liminf_{j\to\infty}\int_0^1 F(\gamma_*)\abs{\gamma_j'}\,d\tau.\]
Then for any fixed $\sigma\in L^2((0,1);\R^2)$ such that $\abs{\sigma}\leq 1$ the weak convergence of derivatives implies that
\[
\liminf_{j\to\infty}\int_0^1 F(\gamma_*)\abs{\gamma_j'}\,d\tau\geq
\liminf_{j\to\infty}\int_0^1 F(\gamma_*)\big(\gamma_j'\cdot\sigma\big)\,d\tau
=\int_0^1 F(\gamma_*)\big(\gamma_*'\cdot\sigma\big)\,d\tau.\]
Taking the supremum over all such $\sigma$ we arrive at the desired inequality,
\[
m_*=\liminf_{j\to\infty}E(\gamma_j)\geq \int_0^1 F(\gamma_*)\abs{\gamma_*'}\,d\tau=E(\gamma_*)\]
so $\gamma_*$ solves the problem \eqref{purple1}.
\qed

\subsection{A parameter regime without bubbling}

Our proof of Theorem \ref{mainiso} does not rule out the possibility that the minimizer $\gamma_*:[0,1]\to\R^2$
passes through either $\ppl$ or $\pmi$ at various instances $t\in (0,1)$. Depending on the`topography' of the graph
of $W$ away from the wells and the size of the constraint value $A_0$, it seems quite possible that one or more such
`bubbles' looping from one of the wells back to itself does indeed occur for the minimizer. In this section, however, we identify a
setting where this cannot occur. This will in particular allow us in the next section to make the connection between such isoperimetric curves
and traveling wave solutions to the Hamiltonian system \eqref{hammy} mentioned in the introduction.

We will show that no bubbling occurs provided the constraint value $\mP$ only differs from that of length-minimizing geodesic
by a small amount:
\bthm\label{nobubbles}
Assume that $W$ satisfies the conditions that held in Theorem \ref{mainiso}.  Let $\gamma_0$ be an unconstrained minimizer
of $E(\gamma)$ among locally Lipschitz curves $\gamma:[0,1]\to\R^2$ satisfying $\gamma(0)=\pmi$ and $\gamma(1)=\ppl.$
Then there exists a number $\e_0>0$ such that for any $\e$ with $0<\abs{\e}<\e_0$, any solution $\gamma_*$ to
\beq
\inf \left\{E(\gamma):\; \gamma:[0,1]\to\R^2\;\mbox{locally Lipschitz},\; \gamma(0)=\pmi,\,\gamma(1)=\ppl,\,\mP(\gamma)=\mP(\gamma_0)+\e\right\}
\label{epvar}
\eeq
satisfies $\gamma_*(t)\not\in\{\pmi,\ppl\}$ for all $t\in (0,1).$ That is, $\gamma_*$ has no bubbles.
\ethm
\proof
Recall from the proof of Theorem \ref{eikonal} that within the balls $B_-$ and $B_+$ of radius $R_-$ and $R_+$ about $\pmi$ and $\ppl$ respectively, the minimizer $\gamma_*$ follows the integral curves
of a vector field. Since there is one vector field associated with each of the two wells we can write  $V_{\pm}(p)=\nabla g_{\pm} +\Lambda_{\pm}$, where $\Lambda_{\pm}$ is the linear vector field depending on the positive
eigenvalues of $D^2 W$ at ${\bf p}_{\pm}$ as in $\eqref{Vfield}$ and $g_{\pm}$ are analytic functions starting at cubic order in a Taylor series about ${\bf p}_{\pm}.$
In the quadratic case where $g_{\pm}\equiv 0$, these integral curves are explicit and one can check that
for some positive radii $r_-\leq R_-$
and $r_+\leq R_+$ such integral curves must exit the balls $B_{r_-}(\pmi)$ and $B_{r_+}(\ppl).$ Since $\nabla g_{\pm}$ represents a small perturbation
of the quadratic setting, the same must be true (in perhaps smaller balls) for the more general case.
That is, there cannot be any homoclinic orbits of these vector fields
lying entirely in small balls about the wells. Alternatively,
one can draw this conclusion by appealing to the Hartman-Grobman Theorem, based on the the fact that the eigenvalues of the linearization of the system $\gamma'=V_{\pm}(\gamma)$ given in \eqref{mupl} have negative real parts.

As a consequence, the only possible bubble around, for instance, $\pmi$ would consist of a curve, say $\tg:[0,T]\to\R^2$ with $\tg(0)=\tg(T)=\pmi$ and
$\abs{\tg(t_1)-\pmi}=r_-$ for some $t_1\in (0,T).$ Invoking \eqref{lbound} such a curve must satisfy the lower energy bound
\beq
E(\tg)\geq
\int_0^{t_1}F(\tg)\abs{\tg\,'}\,dt\geq c_0\int_0^{t_1}\abs{\tg-\pmi}\abs{\tg\,'}\,dt\geq \frac{c_0}{2}\abs{\int_0^{t_1}\frac{d}{dt}\abs{\tg-\pmi}^2\,dt}=\frac{c_0 r_-^2}{2}.
\label{lbgt}
\eeq
In addition to the bubble, any competitor in \eqref{epvar} must also include one curve joining
 $\pmi$ to $\ppl$ and so the total energy of any competitor must have an energy of at least $E(\gamma_0)+\frac{c_0 r_-^2}{2}$.

On the other hand, under these assumptions on $W$, the conformal factor $F=\sqrt{W}$ also satisfies an upper bound of the form
\beq
F(p)\leq c_1\abs{p-\pmi}\quad\mbox{in}\;B_-\;\mbox{for some positive constant}\;c_1.
\label{upperF}
\eeq
We will now build a competitor in the variational problem \eqref{epvar} of the form $\gamma_0\cup\beta_\e$  where $\beta_\e$ is designed so as to satisfy the
constraint $\mP(\beta_\e)=\e$. Specifically we take $\beta_\e$ to consist of the closed semi-circle centered at $\pmi$ of radius $r_\e<R_-$ with vertical diameter.
That is, $\beta_{\e}$ is the union of the two parametric curves $\theta\mapsto \pmi+r_\e(\cos\theta,\sin\theta)$ for $-\pi/2\leq \theta\leq \pi/2$ and $t\mapsto
\pmi+t\,(0,1)$ for $-r_\e\leq t\leq r_\e$.
We determine $r_\e$ through the calculation (made for $\e>0$) that
\[
\e=\mP(\beta_\e)=-\int_0^{\pi}\beta_\e^{(2)}\beta_\e^{(1)}\,'\,d\theta=r_\e^2\int_0^{\pi}\sin^2\theta\,d\theta=\pi\frac{r_\e^2}{2}.
\]
We note that the vertical diameter of the semi-circle $\beta_\e$ does not contribute to the constraint since $\beta_\e^{(1)}\,'=0$ there. If $\e<0$, we reverse
the orientation of $\beta_\e$.

Then we use \eqref{upperF} to compute the upper bound
\beq
E(\beta_\e)\leq c_1\int |\beta_\e||\beta_\e'|\leq c_1r_\e\int\abs{\beta_\e'}=c_1r_\e\left(\pi r_\e+2r_\e\right)=c_1(\pi+2)\frac{2}{\pi}\e,
\label{upen}
\eeq
leading to an upper bound on the energy of the minimizer of \eqref{epvar} of $E(\gamma_0)+c_1(\pi+2)\frac{2}{\pi}\e$.
 Choosing
 \[\e_0=\frac{c_0}{c_1}\left(\frac{\pi}{2+\pi}\right)r_-^2\]
 we see in light of \eqref{lbgt} that for $\abs{\e}<\e_0$ any bubble would lead to a total energy exceeding our upper bound in \eqref{upen} and so is precluded.
 \qed

\section{Application to bi-stable Hamiltonian traveling waves}
Now we return to one of our motivations for considering the degenerate isoperimetric problem: the construction
of traveling wave solutions to the bi-stable Hamiltonian system
\beq
\J u_t=\Delta u -\nabla W_u(u)\label{RPBC}
\eeq
where  $\J$ denotes the symplectic matrix given by
\[
\J=\left(\begin{matrix} 0 & 1\\ -1& 0\end{matrix}\right),
\]
$W:\R^2\to [0,\infty)$ vanishes at two points and  $u:\R^n\to\R^2$, $n\geq 1$. As mentioned in the introduction, \eqref{RPBC} is the Hamiltonian flow associated with
\beq
\int \frac{1}{2}\abs{\nabla u}^2 + W(u)\,dx.
\eeq

In order to view our isoperimetric curves, appropriately reparametrized, as classical traveling wave solutions to \eqref{RPBC}, we will assume in this section that in addition to satisfying the conditions of Theorem \ref{mainiso}, $W$ is a sufficiently smooth function throughout $\R^2.$
One could certainly assume less and then make the connection in the context of weak solutions but we choose not to do so here.
Here we seek a traveling wave solution to \eqref{RPBC} that joins the two minima of $W$ taking the form
\[u(x,t)=U(x_1-\nu t)=U(y)\quad\mbox{for some}\;\nu\in\R\]
where $U(\pm\infty)={\bf p}_{\pm}.$
Such a $U:\R\to\R^2$ would have to satisfy the ODE
\beq
-\nu \J U'=U''-\nabla_uW(U)\quad\mbox{for}\;-\infty<y<\infty,\quad U(\pm\infty)={\bf p}_{\pm}.\label{ODE}
\eeq
As noted previously, in the case of a standing wave where $\nu=0$ in \eqref{ODE}, a solution, the so-called heteroclinic
connection, is already known to exist cf. \cite{ABC,AF,AK,S2}. Our aim is to establish solutions with non-zero wave speed $\nu.$
\subsection{Traveling wave existence under a generic assumption on heteroclinics}

We begin with the observation that \eqref{ODE} has a variational formulation.
To see this, fix any $A\in\R$, and consider the constrained minimization problem given by
\beq
\mu_*:=\inf_{\mG_{A}}H(u)\quad\mbox{where}\;H(u):=\int_{-\infty}^\infty \frac{1}{2}\abs{u'}^2 + W(u)\,dy\label{mu1}
\eeq
and where the admissible set $\mG_{A}$ is defined by
\[
\mG_{A}:=\left\{u:\R\to\R^2:\;u-g\in H^1(\R;\R^2),\;\mP(u)=A\right\}.
\]
In the above definition, $g:\R\to\R^2$ is any smooth function such that $g(y)\equiv \ppl$ for $y$ sufficiently
large and $g(y)\equiv \pmi$ for $y$ sufficiently negative, and
\[
\mP(u):=\,-\int_{-\infty}^\infty u^{(2)} (u^{(1)})'\,dy\quad \mbox{where}\;u=(u^{(1)},u^{(2)}).
\]
Then one has
\bprop\label{lagrange}
Any minimizer of \eqref{mu1}--in fact any critical point--satisfies
\eqref{ODE} for some $\nu\in\R$.
\eprop
\proof Invoking the theory of Lagrange multipliers we can assert the existence of a number $\lambda$ such that any critical
point $u$ satisfies the condition
\[
\delta H(u)=\lambda\,\delta \mP(u)
\]
where $\delta$ refers to the first variation. Since
\[
 \delta H(u;\tilde{u})=\int u'\cdot\tilde{u}'+\nabla W_u(u)\cdot \tilde{u}\,dy\quad\mbox{and}\quad
\delta \mP(u;\tilde{u})=-\int \J u'\cdot \tilde{u}\,dy\quad\mbox{for any}\;\tilde{u}\in H^1(\R),
\]
the result follows with $\lambda$ playing the role of the wave speed $\nu.$\qed
\vskip.1in
Now we recall our degenerate isoperimetric problem:
\beq
m_*:=\inf_{\mS_{A}}E(\gamma)\label{deeppurple}
\eeq
where again $E(\gamma)=\int_0^1 \sqrt{W(\gamma)}\abs{\gamma'}\,dt$ and
\[
\mS_{A}:=\big\{\gamma:[0,1]\to\R^2:\,\gamma\;\mbox{locally Lipschitz},\;\gamma(0)={\bf p}_-,\;\gamma(1)={\bf p}_+,\; \mP(\gamma)=A\big\}
\]
Note that the value of both $E$ and $\mP$ are invariant under reparametrization.

We would like to reparametrize the solution $\gamma_*$ to \eqref{deeppurple} guaranteed by Theorem \ref{mainiso} and then
identify it as a traveling solution to \eqref{RPBC}. However, in the event that $\gamma_*$ possesses bubbles, we would need to
work only with a truncation of $\gamma_*$ consisting of the one piece joining $\pmi$ to $\ppl$. Such a truncation procedure has the disadvantage
of a loss of information: we have no control on how much, if any, of the constraint value $A_0$ is taken up by the truncation. In particular,
it is conceivable that the Lagrange multiplier associated with this truncation vanishes. Then, after the reparametrization
described below, one would only have a heteroclinic orbit, that is a standing wave solution to \eqref{ODE} in which $\nu=0.$

Instead, in the theorem below, we place ourselves in the context of Theorem \ref{nobubbles}, where bubbling is precluded.
Before stating it, we introduce what will be a convenient reparametrization of
any Lipschitz curve $\gamma:[0,1]\to\R^2$ satisfying $\gamma(0)=\pmi$, $\gamma(1)=\ppl$ and $\gamma(t)\not\in\{\pmi,\ppl\}$
for $t\in (0,1)$ and originally parametrized with constant speed $L$:

We let $y: (0,1)\to\R$ be defined through
\beq
y(t):=\frac{1}{\sqrt{2}}\int_{\frac{1}{2}}^t \frac{\abs{\gamma'(\tau)}}{\sqrt{W(\gamma(\tau))}}\,d\tau=\frac{L}{\sqrt{2}}\int_{\frac{1}{2}}^t \frac{1}{\sqrt{W(\gamma(\tau))}}\,d\tau.\label{reparam}
\eeq
Then we define say $U=U(y)$ mapping $\R$ to $\R^2$ via $U(y)=\gamma(t(y))$ so that $U$ will satisfy the condition
\beq
\abs{U'(y)}=\sqrt{2}\sqrt{W(U(y))}.\label{ab}
\eeq
\bthm\label{equivprobs}
For $\e_0$ given by Theorem \ref{nobubbles} and for any non-zero $\e\in (-\e_0,\e_0)$, let $\gamma_*=\gamma_*(t)$ be any minimizer of \eqref{epvar}, which by Theorem \ref{mainiso} must exist and by Theorem \ref{nobubbles} possesses no bubbles.
We take this curve to be parametrized
by constant velocity $\abs{\gamma_*'(t)}=L$, where $L$ is its Euclidean arc length so that $\gamma_*:[0,1]\to\R^2.$
Let $U_{*}:\R\to\R^2$ given by $U_{*}(y):=\gamma_*(t(y))$ be
the reparametrization of $\gamma_*$ given by \eqref{reparam}.
Then we have:\\

\noindent
(i) The function $U_{*}$  lies in $\mG_{A}$ with $A=\mP(\gamma_0)+\e$ and minimizes $H$ within this class, i.e. $H(U_{*})=\mu_*.$
\\
(ii) Let $\mathcal{C}$ denote the set of all heteroclinic connections between $\pmi$ and $\ppl$.
 Then for all non-zero $\e\in (-\e_0,\e_0)$ such that
\beq
\mP(\gamma)\not=\mP(\gamma_0)+\e\quad\mbox{for all}\;\gamma\in \mathcal{C},
\label{bade}
\eeq
the function $U_{*}$ solves \eqref{ODE} with non-zero speed $\nu$.
\ethm
\noindent {\it Proof of (i)}. Given any $u\in \mG_{A}$ with $A=\mP(\gamma_0)+\e$ we use $2ab\leq a^2 + b^2$ to note that
\[
H(u)\geq \sqrt{2}\int_{-\infty}^\infty\sqrt{W(u)}\abs{u'(y)}\,dy=\sqrt{2}E(u).\]
Since $E$ is invariant under reparametrization, we have
\[
\inf_{\mG_{A}}E(u)=\inf_{\mS_{A}}E(u)\]
and consequently $\mu_*\geq \sqrt{2}m_*.$ On the other hand, reparametrizing the minimizer $\gamma_*$ of \eqref{deeppurple} as in \eqref{reparam},
we obtain $U_{*}\in\mG_{A}$ satisfying \eqref{ab}, and so
\[
\sqrt{2}m_*=\sqrt{2}E(U_{*})=H(U_{*})\geq \mu_*,
\]
so it follows that $\mu_*=\sqrt{2}m_*$ and $H(U_{*})=\mu_*.$\\
\noindent {\it Proof of (ii)}. By Proposition \ref{lagrange}, $U_{*}$ solves \eqref{ODE} as well.
Furthermore, $\nu$ cannot be zero since this would make $U_{*}$ a heteroclinic orbit satisfying
$\mP(U_{*})=\mP(\gamma_0)+\e$, violating condition \eqref{bade}.
\qed
\vskip.2in
The condition providing for a non-zero wave speed given by condition \eqref{bade} is clearly generic, in that for it to fail for all non-zero $\e$ in the interval $(-\e_0,\e_0)$,
there would, in particular, have to exist a continuum of heteroclinic orbits connecting $\pmi$ to $\ppl$. However, verifying \eqref{bade} in practice may not
 be so easy. In the next section, we present an alternative set of sufficient conditions
that guarantee the existence of a traveling wave solution to \eqref{ODE} having non-zero speed.

\subsection{Sufficient conditions for existence of a traveling wave}

The goal of this section is to present a set of sufficient conditions under which there are traveling wave solutions with speed $\nu\neq 0$. We will conclude
this section with an explicit example of a potential $W$ meeting these criteria.

 For this purpose, we begin with the following hypotheses on $W\in C^3(\R^2;\R)$, the first two of which have been operable throughout this article:
\begin{enumerate}
\item[(W1)]  $W(p)\ge 0$ and $W(p)=0$ iff $p=\mathbf{p_\pm}$.
\item[(W2)]  $\mathbf{p_\pm}$ are non-degenerate global minima of $W$, ie, $D^2 W(\pm\mathbf{p})\ge \lambda>0$.
\item[(W3)]   There exist constants $R_0,c_0>0$ such that
$$\nabla W(p)\cdot p\ge  c_0  |p|^2
\quad \text{for all $p\in\R^2$ with} \quad |p|\ge R_0.
$$
\end{enumerate}
We should remark that conditions (W1) and (W3) in particular imply the lower bound \eqref{lb} on $F:=\sqrt{W}.$

An important role in what follows will be played by heteroclinic connections, that is, solutions to \eqref{ODE} for which
$\nu=0.$ We recall the notation of Theorem \ref{equivprobs} where we denote the set of all such solutions by $\mathcal{C}$.
We also wish to distinguish those heteroclinic connections that minimize the functional $H$ within the set
$$  \mathcal{S}:=\{ U\in H^1_{loc}(\R;\R^2): \ U(x)\to \mathbf{p}_\pm \ \text{as $x\to\pm\infty$}\}.
$$
We will denote this subset of $\mathcal{C}$ by $\mathcal{C}_0$.
The existence of minimizing heteroclinic connections under hypotheses (W1)--(W3) is well-known (see \cite{AF} for existence under the most general hypotheses).

In order to find traveling waves with nonzero speed, we will look nearby the minimizing heteroclinic connections.  To do this, we must consider the linearized system, or the second variation of the energy $H$, at a minimizing heteroclinic $U_0\in\mathcal{C}_0$, given by
$$  \delta^2H(U_0)[\Phi] := \int_{-\infty}^\infty
       \left[ |\Phi'|^2 + \Phi\cdot D^2 W(U_0)\Phi\right] dy. $$
 This is well defined for $\Phi\in H^1(\mathbb{R}; \mathbb{R}^2)$.  Since any $U_0\in\mathcal{C}_0$ is an (unconstrained) minimizer of $H$, then clearly $\delta^2 H(U_0)[\Phi]\ge 0$ as a quadratic form on $\YY$.
We note that by the translation invariance of this problem on $(-\infty,\infty)$, zero is an eigenvalue of this operator, with $L^2$ eigenfunction $U'_0(y)$. (See Lemma~\ref{spectral} below.)  If, however, zero is an isolated {\em simple} eigenvalue, then the minimizer $U_0$ will be suitably isolated from the other heteroclinics, and we may show that
curves minimizing $H$ in $\mathcal{G}_A$ with $A$ near $\mP(U_0)$ have non-zero wave speed.

\begin{thm}\label{BIG}
Assume (W1)--(W3), and in addition:
\begin{enumerate}
\item[(W4)] There exist neighborhoods $B_\rho(\ppm)$ in which $W$ is either purely quadratic or real analytic and satisfying \eqref{Wexp};
\item[(W5)]  For every $U_0\in\mathcal{C}_0$, zero is an isolated simple eigenvalue of $\delta^2 H(U_0)$.
\end{enumerate}
Then, there exists $\e_*>0$ so that for every
$A\in (\mP(U_0)-\e_*, \mP(U_0)+\e_*)$ with $A\not=\mP(U_0)$, there exists a traveling wave solution $U_A$ with wave speed $\nu=\nu_A\neq 0$.
\end{thm}

Making sure $\e_*<\e_0$ from Theorem \ref{nobubbles}, the hypotheses (W1)-(W4) allow us to apply Theorem \ref{equivprobs}(i) to assert
existence of a minimizer $U_*$ of $H$ within the class $\mathcal{G}_A$ for $A\in (\mP(U_0)-\e_*, \mP(U_0)+\e_*)$ with $A\not=\mP(U_0)$.
Hence, by Proposition \ref{lagrange}, $U_*$ solves the differential equation \eqref{ODE} as well. The new assertion here, however, is that
with the additional assumption (W5), the wave speed is non-zero.

Later in this section, we will present an explicit example of a double-well potential $W$ meeting all the criteria (W1)-(W5).

Noting that the strong assumption (W4) is only used to establish the existence of a minimizer, we can separate the issue of non-zero wave speed
from that of existence by proving:
\begin{prop}\label{simple}
 Assume (W1)--(W3) and (W5).  Then, for any $U_*\in\mC_0$, there exists $\e_1>0$ such that
 if $A\in (\mP(U_*)-\e_1, \mP(U_*)+\e_1)$ with $A\not=\mP(U_*)$ and if $U_A:\R\to\R^2$ attains the minimum of $H$ in $\mG_{A}$, then $U_A$ solves \eqref{ODE} with wave speed $\nu=\nu_A\neq 0$.
 \end{prop}
 Once we establish Proposition \ref{simple}, this will then complete the proof of Theorem \ref{BIG} by choosing $\e_*=\min\{\e_0,\e_1\}$.

\smallskip

In order to prove Proposition \ref{simple} we require the following lemmas involving the minimization problem for $H$ in $\mS$. The first is a restatement of Lemma 2.4 of \cite{ABGui}, on the $H^1$ compactness of minimizing sequences:

\begin{lemma}\label{convlemma} Assume (W1)--(W3).
Let $\{V_n\}\subset \mS$ be a minimizing sequence for $H$.  Then there exist translations $\tau_n\in\R$, a function $U_0\in\mC_0$, and a subsequence of $\{\tau_n\}$ along which\\ $\| V_n(\cdot-\tau_n) - U_0\|_{H^1(\R)}\to 0$ as $n\to\infty.$
\end{lemma}
We remark that while \cite{ABGui} concerns symmetric potentials $W$, the proof of Lemma 2.4 in \cite{ABGui} does not use the symmetry assumption and so applies to any of the $W$ considered
in this article.

Next, we prove a simple {\it a priori} estimate for all heteroclinic solutions $V\in\mC$ (not only the energy minimizers constituting $\mC_0$).

\begin{lemma}\label{apriori}  Assume (W1) and (W3).
There exists a constant $K>0$, depending only on $W$, such that for any heteroclinic $V\in \mC$ one has
\begin{equation}\label{apest}  \| V \|_{L^\infty(\R)} \le K.  \end{equation}
\end{lemma}

\begin{proof}
From hypothesis (W3), we may easily obtain the global bound,
\begin{equation}\label{gb}
\nabla W(p)\cdot p \ge c_0|p|^2 - c_1,
\end{equation}
valid for all $p\in\R^2$.  Define $\varphi(x):=|V(x)|^2-K^2$, for any constant
$K^2> \max\{\frac{c_1}{c_0},|\ppm|\}$.  We calculate
\begin{align*}
 \frac12\varphi''(x) &= |V'|^2 + \nabla W(V)\cdot V \\
   &\ge c_0\left(|V|^2 - \frac{c_1}{c_0}\right)  \\
   &\ge c_0 \varphi.
\end{align*}
Since by the choice of $K$ one has $\lim_{|x|\to\infty} \varphi(x)<0$, the positive part $\varphi_+(x)=\max\{\varphi(x),0\}$ has compact support in $\R$.
Multiplying the inequality above by $\varphi_+$ and integrating, we have:
$$  \int_\R \left[ \frac12 (\varphi'_+)^2  + c_0 \varphi_+^2\right] dx \leq 0,
$$
and hence $\varphi(x)\le 0$ on $\R$.
\end{proof}

The third lemma contains the spectral properties of the linearized operator at any $U_0\in\mC_0$,
$$  L\Phi:= \delta^2 H(U_0)\Phi = -\Phi'' + D^2W(U_0)\Phi, $$
with $\Phi=(\varphi_1,\varphi_2)\in H^1(\R;\R^2)$.  Denote by $(U,V)$ the $L^2(\R;\R^2)$ scalar product.

\begin{lemma}\label{spectral}  Assume $W$ satisfies (W1)--(W3).  Then:
\begin{enumerate}
\item[(a)]
 $L$ is a positive semi-definite linear operator, $(L\Phi,\Phi)\ge 0$ for all $\Phi\in H^1(\R;\R^2)$.
\item[(b)] Zero is an eigenvalue of $L$, with eigenfunction $U'_0\in L^2(\R;\R^2)$.
\item[(c)]  The essential spectrum $\sigma_{ess}(L)\subset [\lambda,\infty)$, with $\lambda$ as in (W2).
\end{enumerate}
\end{lemma}

\begin{proof}
Since $U_0$ is a global minimizer of $H$, it follows that $L=\delta^2 H(U_0)\ge 0$.  By translation invariance, $U'_0$ solves the ODE $LU_0'(x)=0$ on $\R$.  By hypothesis (W2) and the Stable/Unstable Manifold Theorem, $U_0'(x)$ has exponential decay to zero as $x\to\pm\infty$, and thus $U_0'\in L^2(\R;\R^2)$, so zero is an eigenvalue.  Finally, the essential spectrum of $L$ is determined by Weyl's Theorem on locally compact perturbations of linear operators (see \cite{HS}).  Indeed, the essential spectrum of $L$ is the union of the spectra of the constant coefficient operators given by the asymptotic behavior of the potential $D^2 W(U_0(x))$ as $x\to\pm\infty$, $L_-\Phi:= -\Phi'' + D^2 W(\pmi)\Phi$ and $L_+\Phi:= -\Phi'' + D^2 W(\ppl)\Phi$.  These are constant coefficient second order ODEs, and the spectrum of each is bounded below by $\lambda$, by (W2).
\end{proof}

We next show that hypothesis (W5) ensures that minimizing heteroclinics in $\mC_0$ are suitably isolated from heteroclinics with near-minimal energy values.

For simplicity (and without loss of generality) in the following we will choose the wells $\ppm$ of $W$ to be symmetrically placed on the $p_1$-axis. This can be achieved
through a rigid motion of the coordinate system.

\begin{lemma}\label{isolemma}  Assume $W$ satisfies (W1)--(W3) and (W5), and let $U_0\in\mC_0$.  Then there exists $\delta_0>0$ so that if $\{V_n\}$ is any sequence in $\mC$ with
\begin{enumerate}
\item[(a)]
$V_n(\cdot)\neq U_0(\cdot -\tau_n)$ for all $\tau_n\in\R$; and
\item[(b)]
$H(V_n)\to H(U_0)=\min_{U\in\mH} H(U)$,
\end{enumerate}
then
\begin{equation}\label{isolated}  \inf\left\{ \left\| U_0-V_n(\cdot-\tau)\right\|_{L^2(\R)} \, : \ \tau\in\R, \ n\in\N\right\} \ge\delta_0.
\end{equation}
\end{lemma}

\begin{proof}
This lemma is an adaptation of Proposition 5.1 of \cite{ABCP}.  For completeness we provide some details of how to modify the argument for our setting.
Recall that (without loss of generality) we are assuming
$$   \ppm = (\pm 1, 0).  $$

First, let $V\in \mC$ be geometrically distinct from $U_0$; that is, $V(x-\tau)\neq U_0(x)$ holds for every $\tau\in\R$.  Then, there exists $\sigma\in\R$ which attains the minimum value of
$$  \inf_{\tau\in\R} \| U_0(\cdot) - V(\cdot-\tau)\|_{L^2(\R)} =
          \| U_0(\cdot) - V(\cdot-\sigma)\|_{L^2(\R)},
$$
and moreover this translate $V^\sigma(x):=V(x-\sigma)$ satisfies the orthogonality condition,
\begin{equation}\label{Vs}
    \int_\R V^\sigma(x)\cdot U'_0(x)\, dx =0.
\end{equation}
The existence of a minimizing $\sigma$ follows from the asymptotic conditions on $U_0,V$, and
the relation \eqref{Vs} follows by differentiating $\| U_0(\cdot) - V(\cdot-\tau)\|_{L^2(\R)}$ with respect to $\tau$, noting  that $U_0$ itself satisfies the same orthogonality condition,
$$  \int_\R U_0\cdot U'_0\, dx = \frac12\left(|\ppl|^2 - |\pmi|^2\right) =0, $$
in light of the assumed symmetric positioning of $\pmi$ ad $\ppl$.

Let $\Phi=V^\sigma - U_0\in H^1(\R;\R^2)$. Then,
$( \Phi , U'_0)_{L^2} =0$, i.e., $\Phi\in Z:=\text{span}\,\{U'_0\}^\perp$.  Furthermore, by Lemma~\ref{spectral} and (W5), the linearization $L$ is invertible on the closed subspace $Z$.
In particular, its quadratic form is bounded below away from zero: there exists $c_0>0$ for which
\beq
(\Psi,L\Psi)_{L^2}\ge c_0\|\Psi\|_{L^2}^2\quad\mbox{for all}\;\Psi\in Z.\label{bbl}
\eeq
Assume $\{V_n\}\subset\mC$ with $H(V_n)\to H(U_0)$, where we may also assume that each $V_n$ is translated such that
$$
\| V_n - U_0\|_{L^2(\R)}=\inf_{\tau\in\R}\| V_n(\cdot+\tau)-U_0(\cdot)\|_{L^2(\R)}.
$$
Thus, each $V_n$ satisfies the orthogonality relation $(V_n,U'_0)_{L^2}=0$.
Define
$\Phi_n:= V_n- U_0$, and suppose that $\|\Phi_n\|_{L^2}\to 0$, contrary to the statement of Lemma~\ref{isolemma}.

Next we apply a Taylor's expansion to the function $s\mapsto DW(U_0+s\Phi_n)$:  for each $n$ there exists $s_n\in (0,1)$ such that
\begin{align*}
  0=&-\partial _x^2V_n+DW(V_n)=-\partial_x^2U_0-\partial_x^2\Phi_n+DW(U_0+\Phi_n)=\\ &
         -\partial_x^2U_0+DW(U_0) -\partial_x^2\Phi_n  + D^2W(U_0)\Phi_n + \frac12 \frac{d^2}{ds^2}|_{s=s_n}DW(U_0+s\Phi_n)\\
 &= L\Phi_n + \frac12 D^3 W(U_0+s_n\Phi_n)[\Phi_n,\Phi_n].
\end{align*}
To this end, set $t_n:=\|\Phi_n\|_{L^2}\to 0$ and $\tilde \Phi_n:= \Phi_n/t_n\in Z$.  Then, $\tilde\Phi_n$ solves,
\begin{equation}\label{Phi_eq}    L\tilde\Phi_n =-\frac12 t_n\, D^3W(U_0+s_n\Phi_n)[\tilde\Phi_n,\tilde\Phi_n].
\end{equation}

By Lemma~\ref{apriori}, since $V_n\in\mC$ we have a uniform bounds $\|V_n\|_{L^\infty},\|\Phi_n\|_{L^\infty}\le C_1$ for constant $C_1$ independent of $n$.  Using hypothesis (W1), we may conclude that $D^3 W$ is uniformly bounded on bounded sets, and thus we may obtain the uniform estimate:
\begin{align} \label{Lest}  ( \tilde\Phi_n, L\tilde\Phi_n)_{L^2} &=
         -\frac12  \int_\R \sum_{i,j,k}
                \partial_{ijk}W(U+s_n\Phi_n)\tilde\Phi_{n,i}\tilde\Phi_{n,j}\Phi_{n,k} \\
                \nonumber
         &\le C_2 \|\Phi_n\|_{L^\infty}\|\tilde\Phi_n\|_{L^2}^2 \le C_3.
\end{align}
 From this we conclude the uniform bound on the derivatives,
$$  \int_\R  |\tilde\Phi'_n|^2 dx = ( \tilde\Phi_n, L\tilde\Phi_n)_{L^2} - \int_\R D^2W(U_0)[\tilde\Phi_n,\tilde\Phi_n]\, dx
\le C_3 + C_4\|\tilde\Phi_n\|_{L^2}^2=C_3+C_4.
$$
By the Sobolev embedding we conclude that $\|\tilde\Phi_n\|_{L^\infty}\le C_5$ is uniformly bounded, and we may improve the estimate \eqref{Lest},
\begin{align*}
  ( \tilde\Phi_n, L\tilde\Phi_n)_{L^2}
  &= -\frac12 t_n \int_\R \sum_{i,j,k}
                \partial_{ijk}W(U_0+s_n\Phi_n)\tilde\Phi_{n,i}\tilde\Phi_{n,j}\tilde\Phi_{n,k}  \\
    &\le C t_n\|\tilde\Phi_n\|_{L^2}^2\|\tilde\Phi_n\|_{L^\infty} \to 0.
\end{align*}
Since $\tilde\Phi_n\in Z$ and $\|\tilde\Phi_n\|_{L^2}=1$ for all $n$, we arrive at a contradiction, as the quadratic form $( \Phi, L\Phi)_{L^2}$ is strictly positive definite for $\Phi\in Z$, cf. \eqref{bbl}.  In conclusion, \eqref{isolated} must hold.
\end{proof}

\begin{cor}
Assume (W1)--(W3) and (W5).  Then there are only finitely many $V\in\mC_0$ which are geometrically distinct (not translations of each other).
\end{cor}

The corollary follows from Lemma \ref{isolemma}, together with the compactness provided by Lemma \ref{convlemma}.

We are now ready to prove Proposition \ref{simple} which will complete the proof of Theorem \ref{BIG} as well:

\smallskip

\begin{proof}[Proof of Proposition \ref{simple}]
Let $\e_0>0$ be as in Theorem~\ref{nobubbles}, $|\e|<\e_0$, and $U_*\in\mC_0$.
We argue by contradiction, and assume that there exist sequences $\{\e_n\}\to 0$ and $\{U_n\}$, minimizers of $H$ in $\mG_{A_n}$, with
\begin{equation}\label{Pvalues}
A_n:=\mathcal{P}(U_n)=\mathcal{P}(U_*)+\e_n,
\end{equation}
 whose wave speed $\nu_n=0$ for every $n\in\N$.
That is, $U_n\in\mC$ for each $n$.  We note that by \eqref{Pvalues} and the translation invariance of $\mathcal{P}$, each $U_n$ and $U_*$ are geometrically distinct elements of $\mC$, that is $U_n(\cdot-\tau)\neq U_*$ for any $\tau\in\R$.

The construction in the proof of Theorem~\ref{nobubbles} provides an upper bound on the energies,
$$   H(V_n) = \sqrt{2}E(V_n) \le \sqrt{2}\left[ E(U_*) + E(\beta_{\e_n})\right]
                          \le \sqrt{2}\left[ E(U_*) + c_1(\pi+2){\frac{2}{\pi}}\e_n\right]
                           = H(U_*) + c_2\e_n,  $$
where $\beta_{\e_n}$ is the bubble whose energy is estimated in \eqref{upen}.
Thus, $\{U_n\}_{n\in\N}$ is a sequence of heteroclinic connections with
$H(U_n)\to H(U_*)$, and thus it is a minimizing sequence (without area constraint) for $H$ in $\mS$.
Applying Lemma~\ref{convlemma}, there exists $U_0\in \mC_0$ and translations $\tau_n\in\R$, so that (along a subsequence)
\begin{equation}\label{contra}
U_n(\cdot-\tau_n)\to U_0\quad \text{in $H^1(\R;\R^2)$.}
\end{equation}
But $U_n$ satisfies the hypotheses of Lemma~\ref{isolemma}, and this contradicts \eqref{isolated}, so there must exist $\e_1$ as claimed in the statement of the Proposition.
\end{proof}

The simplicity of the ground state eigenvalue of $\delta^2H(U_0)$ is generic, but it is possible to identify a class of specific examples for which it must hold.  For instance, we may impose an additional condition on $W$:
\begin{enumerate}
\item[(W6)]  $\pmi$ and $\ppl$ lie on the $p_1$-axis and $W(p_1,p_2)\ge W(p_1,0)$, for all $p=(p_1,p_2)\in\R^2$.
\end{enumerate}
A simple example which satisfies all of the conditions (W1)--(W6) may be constructed in the form $W(p)= w(p_1) + p_2^2$, with $w(p_1)$ an even $C^2$ function with $w(p_1)=(p_1-1)^2$ for $p_1\ge\frac12$, and $w(p_1)>\frac14$ for $|p_1|<\frac12$.

\begin{prop}\label{simple2}
If $W$ satisfies (W1)--(W3) and (W6), and $U_0$ is any  minimizer of $H$ in $\mathcal{S}$, then zero is a simple eigenvalue of $\delta^2H(U_0)$.
\end{prop}

\begin{proof}
For any $U=(u_1,u_2)\in \YY$, $H(U)\ge H(u_1,0)$, and so the minimizer must have the form $U_0=(u_0,0)$. By (W6), $W_{p_2}(U_0(x))=0$ and $W_{p_2p_2}(U_0(x))\ge 0$ for all $x\in\R$.  Moreover, the Euler-Lagrange equations are
$$   -u''_0(x) + W_{p_1}(u_0(x),0) =0,  \quad  W_{p_2}(u_0(x),0) =0.  $$
From the second equation, $W_{p_2 p_1}(u_0(x),0)=0$ (almost everywhere) along the heteroclinic.

From Lemma~\ref{spectral}, the second variation of the energy about $U_0$ is non-negative,
$\delta^2 H(U_0)[\Phi,\Phi]\ge 0$ for all $\Phi\in H^1(\mathbb{R},\mathbb{R}^2)$.  In fact, by the choice of potential,
\begin{align*}
\delta^2 H(U_0)[\Phi,\Phi] &= \int\left[ |\Phi'|^2 + \Phi^t D^2 W(U_0)\Phi \right] dx \\
&= \int \left[ (\varphi'_1)^2 + W_{p_1 p_1}(u_0,0) \varphi_1^2\right] dx +
    \int \left[ (\varphi'_2)^2 + W_{p_2 p_2}(u_0,0) \varphi_2^2\right] dx \\
    &=: (\varphi_1, L_1\varphi_1) + (\varphi_2, L_2\varphi_2),
\end{align*}
defining the linear operators with respect to the $L^2$ scalar product.  Now, both $L_1\ge 0$ and $L_2\ge 0$, as we may consider variations of $U_0$ with only one coordinate nonzero.  By Lemma~\ref{spectral} the essential spectra of $L_1, L_2$ are bounded away from zero.
For $L_1$, $\varphi_1=u'_0$ is an $L^2$ eigenfunction.  Moreover, it is the lowest eigenvalue and thus has a sign $u'_0\ge 0$, and is therefore simple.  For $L_2$, $W_{p_2 p_2}(U_0(x))\ge 0$, so
$(\varphi_2, L_2\varphi_2)=0$ if and only if $\varphi_2$ is constant, hence zero, and thus the operator $L_2$ is invertible.  Since (by Lemma~\ref{spectral}) the essential spectrum is bounded away from zero (and zero cannot be an eigenvalue), there exists a constant $\mu>0$ with $\sigma(L_2)\subset [\mu,\infty)$.
In conclusion, $\delta^2 H(U_0)\ge 0$ and has an isolated simple zero eigenvalue, with eigenfunction $U'_0$.
\end{proof}

\subsection{Non-existence of traveling waves with large wave speeds}
Finally, we investigate the question of whether there is some restriction on the possible wave speeds $\nu$ for traveling wave solutions of \eqref{ODE}.  It is typical for traveling waves of Hamiltonian systems to have some speed limit:  for example, the cubic defocussing Nonlinear Schr\"odinger (NLS) equation has nonconstant traveling waves (``dark solitons'') with speeds $|\nu|
\le\sqrt{2}$.  (See \cite{Chiron} for a discussion of traveling waves in the 1D NLS equation.)
As in the NLS case, we expect traveling waves to exist only for $\nu$ in an interval around zero, and that the critical speed should be determined by the linearized equations at the two wells of $W$.  A complete determination of the possible wave speeds, as well as the relationships between the speed $\nu$ (the Lagrange multiplier in the variational problem \eqref{mu1}) and the area constraint $\mP(U)$, remain interesting open questions.

First, we consider the case that for $|p-\mathbf{p}_\pm|<r_\pm$ with $r_\pm>0$, $W:\R^2\to\R$ is given by a quadratic of the form
\beq
W(p)=(p-{\bf p}_\pm)^T\,H_\pm (p-{\bf p}_\pm),\label{SL1}
\eeq
where $H_\pm$ are a pair constant, real, symmetric, positive definite $2\times 2$ matrices.
Finding a necessary condition for the speed $\nu$ of a traveling wave is a local computation, and so we restrict our attention to one well, which we place at the origin, $\mathbf{p}_+=0$.
We denote by $\lm_1$ and $\lm_2$ the two positive eigenvalues of the matrix $H_+$ and express all points in $\R^2$ and all curves in a neighborhood of that well
with respect to the orthonormal basis $\{v_1,v_2\}$ of eigenvectors of $H_+$, so that
we may assume $W$ takes the form
\[
W(p)=\lm_1 p_1^2+\lm_2 p_2^2, \qquad |p|\le r.
\]
\begin{prop}\label{SLprop1}
Under the above hypothesis on $W$, if there exists a heteroclinic traveling wave solution to \eqref{ODE}, then $\nu^2\le 2(\lambda_1+\lambda_2)^2$.
\end{prop}
\begin{proof}
A heteroclinic traveling wave solution to \eqref{ODE} must remain inside the ball of radius $r_+$ centered at $\ppl$ for a semi-infinite time interval, which we may take to be $x\in (0,\infty)$.  Thus, for $x>0$, we have a solution of
\beq
-\nu \J U'=U''-2H_+ U\quad\mbox{for}\; \lim_{x\to\infty}U(x)=0.\label{SL2}
\eeq
This is a constant-coefficient linear system, which may be solved explicitly.  By converting \eqref{SL2} for $U=(u_1,u_2)$ to an equivalent first-order system of four equations, $Z'=MZ$,
$$
Z=\begin{bmatrix}  z_1 \\ z_2 \\z_3 \\ z_4
\end{bmatrix}=\begin{bmatrix} u_1-\ppl^{(1)} \\ u_2-\ppl^{(2)} \\ u'_1 \\ u'_2 \end{bmatrix}, \qquad
M= \begin{bmatrix}  0 & 0& 1 & 0 \\ 0 & 0 & 0 & 1
               \\ 2\lm_1 & 0 & 0 & -\nu \\  0 & 2\lm_2 & \nu & 0
\end{bmatrix},
$$
any solution is a linear combination of vectors of the form $Z(x)=e^{\mu x} Z_0$, with $\mu$ an eigenvalue of $M$ with eigenvector $Z_0$.  A simple calculation reveals that
the eigenvalues $\mu$ of $M$ satisfy
$$  \mu^2 = \frac12\left[ -\beta \pm \sqrt{\beta^2 -16\lm_1\lm_2}\right], \qquad
\beta:= \nu^2 - 2(\lm_1 + \lm_2),
$$
which leads us to the following parameter regimes in $\nu$:
\begin{enumerate}
\item[(1)] $\boxed{0\le \nu^2 \le 2(\lambda_1-\lambda_2)^2}$.  Then,
$0<4\lambda_1\lambda_2 \le 2(\lm_1+\lm_2)-\nu^2$, and hence $\beta<0$ and $\beta^2-16\lm_1\lm_2>0$.  Thus all eigenvalues $\mu$ are real and nonzero, and at least one is negative.  Traveling waves can exist for such $\nu$, and if they exist they do not spiral.

\item[(2)]  $\boxed{2(\lambda_1-\lambda_2)^2 < \nu^2 < 2(\lambda_1+\lambda_2)^2}$.
In the case of strict inequality, $\beta^2<16\lm_1\lm_2$, and thus $\mu^2=\frac12[\beta\pm i\gamma]$ has a nonzero imaginary part $\gamma>0$.  For such $\nu$, we have two pairs of complex conjugate eigenvalues $\mu$, two of which must have negative real part, and thus correspond to exponentially decaying spiral trajectories.  In this regime, traveling waves can exist.

\item[(3)] $\boxed{\nu^2\ge 2(\lambda_1+\lambda_2)^2}$.  For these $\nu$, we have $\beta\ge 4\lambda_1\lambda_2>0$, and hence $\mu^2<0$ yields purely imaginary eigenvalues.  When the inequality is strict, this would imply that while the traveling wave resides inside the ball of radius $r$ around the well the solutions are purely oscillatory, which contradicts the condition $U(x)\to 0$ as $x\to\infty$.  In the case of equality, we have a pair of pure imaginary, multiplicity-two eigenvalues, and again there are no solutions which decay as $x\to\infty$. Thus, there can be no heteroclinic traveling wave solutions for $\nu$ which are this large.
\end{enumerate}
\end{proof}

\medskip

When the potential $W(p)$ is not exactly a quadratic we nevertheless expect that the permissible wave speeds will be determined by the quadratic part of $W$, as in Proposition~\ref{SLprop1}.
However, in the parameter regime $\nu^2 >2(\lambda_1+\lambda_2)^2$ the linearized traveling wave equations produce a center around the well, corresponding to distinct purely imaginary eigenvalues of the matrix $M$.  As is well known (see\cite{CL}) centers are not structurally stable, and nonlinear perturbations of a center can produce spirals.  One approach to rule out solutions which decay to the wells as $x\to\pm\infty$ is to require the nonlinear perturbation of the quadratic in $W$ to be very small.

\begin{prop}\label{SLprop2}
Assume that in neighborhoods $B_{\rho_\pm}(\mathbf{p}_\pm)$ of the wells $\mathbf{p}_\pm$, we have in local coordinates centered at ${\bf p}_\pm$:
\begin{equation}\label{SL10}
W(p) = \lambda_{1,\pm}^2 p_1^2 + \lambda_{2,\pm}^2 p_2^2 + G(p),
\end{equation}
where $\nabla_p G(p)=g(p)$ with $|g(p)|\le C|p|^q$, and $q\ge 3$.

Then, if $\nu^2 > \min\{2(\lambda_1^+ +\lambda_2^+)^2, 2(\lambda_1^- +\lambda_2^-)^2\}$, then there is no heteroclinic traveling wave with finite energy $H(U)<\infty$ and speed $\nu$.
\end{prop}

\begin{rmrk} \label{SLrem}   \rm
The hypothesis $g(p)\le C|p|^q$ can be integrated to obtain a corresponding bound on $G(p)$, with $\nabla_p G(p)=g(p)$ and $G(0)=0$.  Indeed,
$$  |G(p)|= \left|\int_0^1 \frac{d}{ dt} G(tp)\, dt\right|
      \le \int_0^1 |g(tp)|\, |p|\, dt \le C|p|^{q+1}\int_0^1 t^q\, dt = \frac{C}{ q+1} |p|^{q+1}.
 $$
Thus, the residual term in the potential must be of at least fourth order in $p$ for Proposition~\ref{SLprop2} to apply.
\end{rmrk}

\begin{proof}
Assume there is a solution $U$ of \eqref{ODE} of finite energy $H(U)<\infty$.  We restrict our attention to the neighborhoods of the wells; take $\mathbf{p}_+=0$, and assume $U(x)\in B_{\rho_+}(0)$ for all $x\in [x_0,\infty)$.  Without loss, we assume $x_0=0$.  By the boundedness of the energy and the hypothesis on $W$, this implies the $H^1([0,\infty))$ norm of $U$ is finite.

Note that we can represent the nonlinear term $g(U)$ in the form,
$$ g(U) = \frac{1}{ |U|^2} \begin{bmatrix}  [U\cdot g(U)] & [-U^\perp\cdot g(U)] \\
                                           [U^\perp\cdot g(U)] &  [U\cdot g(U)]
                 \end{bmatrix}\begin{pmatrix} u_1 \\ u_2\end{pmatrix}
                 =: h(U)U,
$$
with two-by-two real matrix $h(U)$, which satisfies $|h(U)|\le C |U|^{q-1}$.
Hence,
\begin{equation}\label{SL12}
\int_1^\infty |h(U)|\, dx \le C\int_1^\infty |U(x)|^{q-1}\, dx
     \le C \rho_+^{q-3} \int_1^\infty |U(x)|^2\, dx <\infty.
\end{equation}

Using the notation from Proposition~\ref{SLprop1} we write \eqref{ODE} as a first order system, of the form,
\begin{equation}\label{SL13}   Z'(x) = (M - H(x))\, Z(x), \qquad H(x):= C=\left[
\begin{array}{c|c}
0  & 0 \\ \hline
 0 & h(U(x))
\end{array}\right],
\end{equation}
with $H(x)$ composed of 2$\times$2 blocks.  Since $\nu^2>2(\lambda_1^+ +\lambda_2^+)^2$, by Proposition~\ref{SLprop1} the eigenvalues of $M$ are distinct and purely imaginary.  By Levinson's Theorem (see Theorem 8.1 and problem 8.29 of \cite{CL},) for each eigenvalue $\mu_j$ and associated eigenvector $\xi_j$, $j=1,2,3,4$ of $M$, there exists a (complex vector-valued) solution $Z_j(x)$ of \eqref{SL13} with
$\lim_{x\to\infty} Z_j(x)e^{-\mu_j x} = \xi_j$, $ j=1,2,3,4$.
  In particular, each $Z_j(x)$ is purely oscillatory, and each is bounded away from zero for  $x\in [0,\infty)$.
The solution $Z(x)$ corresponding to the traveling wave $U(x)$ must be a linear combination of these, so there exist constants $c_j$ with
$$   Z(x)=\sum_{j=1}^4 c_j Z_j(x) \to 0 \quad\text{as $x\to\infty$.}  $$
Since the $Z_j(x)$ form a basis for the solution space on $[0,\infty)$, and none of the $Z_j(x)$ vanishes for $x\to\infty$, each $c_j=0$ and thus $Z(x)=0$ is the only possible solution of \eqref{SL13}, and hence $U(x)=0$ is the only solution to \eqref{ODE} when $\nu^2>2(\lambda_1^+ +\lambda_2^+)^2$.

A similar argument applies in a neighborhood $B_{\rho_-}(\mathbf{p_-})$ when $\nu^2>2(\lambda_1^- +\lambda_2^-)^2$, and thus the proposition is proven.
\end{proof}

\begin{rmrk} \rm
The above approach to nonexistence via the Levinson Theorem requires some integrability assumption on the traveling wave $U$.  In fact, if we replace the finite energy assumption on $U$ by the stronger decay hypotheses $U-\ppl\in W^{1,1}(0,\infty)$ and $U-\pmi\in W^{1,1}(-\infty,0)$, then we could admit more general perturbations from the quadratic well of the form \eqref{SL10} but with $q\ge 2$, that is $G(p)=O(|p|^3)$. (See Remark~\ref{SLrem}).
On the other hand, it is not obvious how to obtain solutions in $W^{1,1}$, either from the boundedness of energy and momentum $\mP$ or via {\it a priori} estimates on solutions, while the finite energy assumption made in Proposition~\ref{SLprop2} is natural given our existence results.
\end{rmrk}

There is another special class of potentials $W$ for which a necessary bound on the wave speed $\nu$ may be derived, without assuming any specific rate of decay of $U$ to the wells $\ppm$, only that $U(x)\to\ppm$ and $U'(x)\to 0$ as $x\to\pm\infty$.  Let us assume that inside a ball of radius $r$ around $p_+=0$,
\beq\label{SL3}
W(U) = \mathcal{G}(|U|^2) = \lambda|U-\ppl|^2 +  G(|U-\ppl|^2),
\eeq
with $G:\mathbb{R}\to\mathbb{R}$ a differentiable function with $G'(t)=o(1)$, $G(t)=o(t)$ for $t$ near zero.
\begin{prop}\label{SLprop3}
Under the hypothesis \eqref{SL3} on $W$,  if there exists a heteroclinic traveling wave with speed $\nu$, then $\nu^2 \le  8\lambda^2$.
\end{prop}

\begin{proof}
Under the assumption \eqref{SL3}, the equation \eqref{ODE} is of nonlinear Schr\"odinger type, and the speed limit may be obtained following the method of Mari\c{s} \cite{Maris} (see also Chiron \cite{Chiron}.)  The proof is based on various exact integrals of the ODE,
\beq\label{SL4}
-U'' + 2\lambda U + g(|U|^2)U = \nu\J U'
\eeq
where $g(t)=2G'(t)$, and
we are assuming that $U(x)$, $x\in (0,\infty)$, lies in the ball of radius $r$ centered at the origin.    Taking the scalar product of the equation \eqref{SL4} with $U'$ we obtain the usual energy integral of the equation,
$$
  0=- U''\cdot U' + 2\lambda U\cdot U' + g(|U|^2)U\cdot U'
 =\frac{d}{ dx} \left( -\frac12 |U'|^2 + \lambda |U|^2 + G(|U|^2)\right).
$$
Integrating, and evaluating the constant as $x\to\infty$, we obtain the identity,
\beq\label{SL5}
\frac12 |U'|^2 = \lambda |U|^2 + G(|U|^2).
\eeq
Next, we take the dot product of \eqref{SL4} with $\J U$, and integrate as above to obtain:
\beq\label{SL6}
U'\cdot \J U = -\frac{\nu}{ 2}|U|^2.
\eeq
Taking the dot product of \eqref{SL4} with $U$ this time produces the relation,
$$
- U''\cdot U + 2\lambda |U|^2 + g(|U|^2)|U|^2 = -\nu U'\cdot\J U
$$
and so with \eqref{SL6} we have:
\beq\label{SL7}
U''\cdot U = \left(2\lambda-\frac{\nu^2}{ 2} + g(|U|^2)\right)|U|^2
\eeq

Now we define $\eta(x):=|U(x)|^2$ and derive an equation for $\eta$,
\begin{align*}
\frac12\eta'' &= |U'|^2 +  U''\cdot U \\
&= \left(4\lambda - \frac{\nu^2}{ 2} + g(|U|^2)\right)|U|^2 + 2G(|U|^2) \\
&= \left(4\lambda - \frac{\nu^2}{ 2}\right) \eta + R(\eta),
\end{align*}
where $R(\eta)=g(\eta)\eta + 2G(\eta)$.  This equation for $\eta$ has an energy integral, obtained by multiplying by $\eta'$ and integrating,
\beq\label{SL8}
\frac14 |\eta'|^2 + \frac12\left(\nu^2 - 8\lambda\right)\eta^2 = \mathcal{R}(\eta),
\eeq
where we have evaluated the constant of integration at $x=+\infty$, and with $\mathcal{R}'(\eta)=R(\eta)$, $\mathcal{R}(0)=0$.  By the hypotheses on $g,G$, $\mathcal{R}(\eta)=o(\eta^3)$ for small $\eta$.  So if $\nu^2>8\lambda$, this leads to a contradiction in case there exists $U(x)$ with $\eta(x)=|U(x)|^2\to 0$ as $x\to\infty$.
\end{proof}

\noindent
{\bf Acknowledgements.} S.A. and L.B. are supported by an NSERC (Canada) Discovery Grant.
 P.S. wishes to acknowledge the support of the National Science Foundation
 through D.M.S. 1362879 and D.M.S. 1101290.


\begin{thebibliography}{100}

\bibitem{ABCP} Alama, S.; Bronsard, L.; Contreras, A.; Pelinovsky, D. Domain walls in the coupled Gross-Pitaevskii equations.
{\it Arch. Ration. Mech. Anal.}  {\bf 215} (2015), 579-610.

\bibitem{ABGui} Alama, A.; Bronsard, L.; Gui, C. Stationary layered solutions in $\R^2$ for an Allen--Cahn system with multiple well potential. {\it Calc. Var. Partial Differential Equations.}
{\bf 5} (1997), no. 4, 359--390.

\bibitem{ABC} Alikakos, N.; Betel\'u, S.; Chen, X. Explicit stationary solutions in multiple-well dynamics and
nonuniqueness of interfacial energy densities. {\it European. J. Appl. Math.} {\bf 17}, (2006), no. 5, 525-556.

\bibitem{AF} Alikakos, N.; Fusco, G. On the connection problem for potentials with several global minima. {\it Indiana Univ. Math. J.}
{\bf 57}, (2008), no. 4, 1871-1906.

\bibitem{AK} Alikakos, N.; Katzourakis, N. Heteroclinic travelling waves of gradient diffusion systems. {\it Trans. A.M.S.} {\bf 363} (2011), no.3,
1365-1397.

\bibitem{AW} Aronson, D.; Weinberger, H. Multidimensional nonlinear diffusion arising in population genetics.  {\it Adv. in Math.} {\bf 30}, (1978), no. 1,
33-76.

\bibitem{CJQW} Carroll, C.; Jacob, A.; Quinn, C.; Walters, R. The isoperimetric problem on planes with density. {\it Bull. Austr. Math. Soc.} {\bf 78} (2008), 177-197.

\bibitem{CMV} Ca\~{n}ete, A.; Miranda Jr., M.; Vittone, D. Some isoperimetric problems in planes with density. {\it J. Geom. Anal.} {\bf 20} (2010), no. 2, 243-290.

\bibitem{Chiron} Chiron, D. Travelling waves for the nonlinear Schr\"odinger equation with general nonlinearity in dimension one. {\it Nonlinearity} {\bf 25} (2012), no. 3,
813-850.

\bibitem{CL} Coddington, E.A.; Levinson, N. {\it Theory of ordinary differential equations}, McGraw-Hill, New York-Toronto-London, 1955.

\bibitem{DDNT} Dahlberg, J.; Dubbs, A.; Newkirk, E.; Tran, H. Isoperimetric regions in the plane with density $r^p$. {\it New York J. Math.} {\bf 15}, (2010), 31-51.

\bibitem{FM} Fife, P.; McLeod, B. The approach of solutions of nonlinear diffusion equations to travelling front solutions. {\it Arch. Ration. Mech. Anal.} {\bf 65},
(1977), no. 4, 335-361.

\bibitem{HS} Hislop, P.D.; Sigal, I.M.  {\it Introduction to spectral theory: with applications to Schr\"odinger operators}. Applied Mathematical Sciences, Vol. 113. Springer, New York, 1996.

\bibitem{H} Howe, S. The log-convex density conjecture and vertical surface area in warped products. To appear in {\it Adv. Geom.}

\bibitem{Maris} Mari\c{s}, M. Nonexistence of supersonic traveling waves for nonlinear Schr\"odinger equations with nonzero conditions at infinity. {\it SIAM J. Math. Anal.} {\bf 40},
(2008), no. 3, 1076-1103.


\bibitem{M} Muratov, C. A global variational structure and propagation of disturbances in reaction-diffusion systems of gradient type. {\it Discrete Contin.
Dyn. Syst. Ser. B}. {\bf 4}, (2004), no. 4, 867-892.

\bibitem{R} Reinick, J. Traveling wave solutions to a gradient system. {\it Trans. A.M.S.} {\bf 307}, (1988), 535-544.

\bibitem{RCBM} Rosales,  C.; Ca\~{n}ete, A.; Bayle, V.; Morgan, F. On the isoperimetric problem in Euclidean space with density. {\it Calc. Var. Partial Differential Equations}. {\bf 31}, (2008), 27-46.

\bibitem{S1} Sternberg, P. The effect of a singular perturbation on nonconvex variational problems. {\it Arch. Ration. Mech. Anal.} {\bf 101}, (1988), no. 3,
209-260.

\bibitem{S2} Sternberg, P. Vector-valued local minimizers of nonconvex variational problems. {\it Rocky Moutain J. Math.}
{\bf 21}, (1991), no. 2, 799-807.


\end{thebibliography}
\end{document}